\PassOptionsToPackage{unicode}{hyperref}
\PassOptionsToPackage{hyphens}{url}
\documentclass{article}
\usepackage{amsmath,amssymb}
\usepackage{iftex}
\ifPDFTeX
  \usepackage[T1]{fontenc}
  \usepackage[utf8]{inputenc}
  \usepackage{textcomp} 
\else 
  \usepackage{unicode-math} 
  \defaultfontfeatures{Scale=MatchLowercase}
  \defaultfontfeatures[\rmfamily]{Ligatures=TeX,Scale=1}
\fi
\usepackage{lmodern}
\ifPDFTeX\else
\fi
\IfFileExists{upquote.sty}{\usepackage{upquote}}{}
\IfFileExists{microtype.sty}{
  \usepackage[]{microtype}
  \UseMicrotypeSet[protrusion]{basicmath} 
}{}
\makeatletter
\@ifundefined{KOMAClassName}{
  \IfFileExists{parskip.sty}{%
    \usepackage{parskip}
  }{
    \setlength{\parindent}{0pt}
    \setlength{\parskip}{6pt plus 2pt minus 1pt}}
}{
  \KOMAoptions{parskip=half}}
\makeatother
\usepackage{xcolor}
\usepackage[margin=1.2in]{geometry}
\usepackage{graphicx}
\makeatletter
\def\maxwidth{\ifdim\Gin@nat@width>\linewidth\linewidth\else\Gin@nat@width\fi}
\def\maxheight{\ifdim\Gin@nat@height>\textheight\textheight\else\Gin@nat@height\fi}
\makeatother
\setkeys{Gin}{width=\maxwidth,height=\maxheight,keepaspectratio}
\makeatletter
\def\fps@figure{htbp}
\makeatother
\newlength{\cslhangindent}
\newlength{\csllabelwidth}
\newlength{\cslentryspacingunit} 
 {
  \setlength{\parindent}{0pt}
  \ifodd #1
  \let\oldpar\par
  \def\par{\hangindent=\cslhangindent\oldpar}
  \fi
  \setlength{\parskip}{#2\cslentryspacingunit}
 }%
 {}
\usepackage{calc}

\usepackage{datetime}
\usepackage{float}
\usepackage{amsthm}
\usepackage{nonfloat}
\usepackage{algorithm}
\usepackage{algpseudocode}
\usepackage{booktabs}
\usepackage{graphicx}
\usepackage{enumitem}
\usepackage[title]{appendix}

\newtheorem{theorem}{Theorem}
\newtheorem{lemma}{Lemma}

\theoremstyle{remark}
\newtheorem{remark}{Remark}

\algsetblock[Name]{Start}{End}{5}{1cm}
\ifLuaTeX
  \usepackage{selnolig}  
\fi
\IfFileExists{bookmark.sty}{\usepackage{bookmark}}{\usepackage{hyperref}}
\IfFileExists{xurl.sty}{\usepackage{xurl}}{} 
\hypersetup{
  hidelinks,
  pdfcreator={LaTeX via pandoc}}

\author{}
\date{\vspace{-2.5em}}

\begin{document}

\newcommand{\AppLemmasOneToFourProofs}{Appendix 1}
\newcommand{\AppLemmaOneOneProof}{Appendix 1.1}
\newcommand{\AppLemmaOneTwoProof}{Appendix 1.2}
\newcommand{\AppLemmaOneThreeProof}{Appendix 1.3}
\newcommand{\AppLemmaOneFourProof}{Appendix 1.4}

\newcommand{\AppProcessDefinition}{Appendix 2}

\newcommand{\AppConvergenceRateProofs}{Appendix 3}
\newcommand{\AppConvergenceRateProofOne}{Appendix 3.1}
\newcommand{\AppConvergenceRateProofTwo}{Appendix 3.2}
\newcommand{\AppConvergenceRateProofThree}{Appendix 3.3}

\newcommand{\AppHessianInverseBoundness}{Appendix 4}

\newcommand{\AppVarianceCalculation}{Appendix 5}

  \newcommand{\revision}[1]{{\color{red} #1}} 

\title{Estimation of service value parameters for a queue with unobserved balking\footnote{To appear in Queueing Systems: Theory and Applications}}
\author{Daniel Podorojnyi~and Liron Ravner\footnote{Department of Statistics, University of Haifa, Israel}}

\maketitle
	
	\vspace{01cm}
\begin{abstract}

In Naor’s model \cite{Naor1969}, customers decide whether or not to join a queue after observing its length. This work considers a variation in which customers are heterogeneous in their service value (reward) $R$ from completed service and homogeneous in the cost of staying in the system per unit of time. It is assumed that the values of customers are independent random variables generated from a common parametric distribution. The manager observes the queue length process, but not the balking customers. Assuming that the distribution of $R$ admits a known parametric form, a Maximum Likelihood Estimator based on the queue length data is constructed for the underlying parameters of $R$. We provide verifiable conditions for which the estimator is consistent and asymptotically normal. The estimation procedure is further leveraged to construct a dynamic pricing scheme that estimates the revenue maximizing admission price by iteratively updating the price using the estimated parameters. The performance of the estimator and the pricing algorithm are studied through a series of simulation experiments.

\end{abstract}

\section*{Acknowledgements}
This work was supported by the Israel Science Foundation (ISF), grant no. 1361/23. The authors would like to thank the anonymous reviewers and associate editor for their detailed comments and suggestions. We are also grateful to Alexander Goldenshluger for his helpful feedback on early versions of this work.

\section{Introduction}\label{introduction.}\label{sec:intro}

Economic models of congested service systems typically involve population-level parameters that govern demand and, in turn, optimal admission pricing. Two fundamental parameters are the individual’s valuation of the service and their sensitivity to waiting time, which together determine the maximum delay a customer is willing to tolerate. These parameters are generally unobserved by the system administrator and may vary substantially across customers. Such heterogeneity is naturally modeled by treating customer valuations as random variables drawn from a population distribution.

A common approach is to assume a parametric form for this distribution and to estimate its parameters using observable system data, such as queue-length trajectories, service times, and interarrival times. These estimates can then be used to address questions of optimal system design and control. For instance, one may compute a revenue-maximizing admission price based on the inferred demand characteristics. Alternatively, social welfare maximization, often of interest to government agencies or non-profit organizations, such as subsidized meal centers, also relies on understanding the underlying utility parameters of customers. In contrast, private service providers, such as restaurant owners, may focus primarily on profit maximization.

In this work, we study a variant of Naor’s classical observable-queue model \cite{Naor1969}, in which customer service values are random and follow a parametric distribution $F_\theta$. Customers observe the queue length upon arrival and decide whether to join or balk based on their individual service value, inducing a random balking threshold. The system administrator observes only the effective queue-length process, as balking customers leave without being recorded. Using this partial information, we construct a maximum likelihood estimator (MLE) for the parameters of the service-value distribution and analyze its asymptotic properties under suitable regularity conditions on the parametric family. 

The estimation procedure is leveraged to develop an iterative pricing algorithm aimed at revenue maximization. We assume that the system operator can set an admission price and wishes to maximize the per unit of time revenue. The operator knows the parametric form of the reward distribution $F_\theta$, but does not observe the lost demand stemming from balking customers. For any given price the operator computes the MLE after observing the queue length for a duration of time. This is then used to update the price with a new an estimator for the revenue maximizing price. This process is repeated with increasingly large data sets, which ensures convergence to the true parameters due to the aformentioned asymptotic guarantees.

\subsection{Statistical model and objective}\label{sec:model}

Customers arrive at a single first-come-first-served (FCFS) queue according to a stationary Poisson process with rate $\lambda$. The service times are independently, identically, and exponentially distributed with parameter $\mu$. Customers are heterogeneous and strategically decide whether to join or balk after observing the queue length. In Naor's model (see \cite{Naor1969}, \cite{Larsen1998} and\cite{book_HH2003}), a customer's benefit from completed service is $R$. The cost to a customer for staying in the system (either while waiting or being served) is $C$ per unit of time spent in the system (sojourn+service time). The system charges an admission price $p$ from every joining customers. Customers are risk-neutral, meaning they maximize the expected value of their net benefit. From a public (social) point of view, the utility functions of individual customers are identical and additive. Thus, given a value $R$, every customer has an individual optimal threshold strategy that dictates whether to join or balk at the observed queue length. A decision to join is irrevocable, and reneging is not allowed. A customer who balks leaves the system and never returns.  The service values of customers are iid with a common random variable $R\sim F_\theta$, where $F_\theta$ is a cumulative distribution function (cdf) and $\theta\in\mathbb{R}^n$ is a parameter that fully determines the distribution (see \cite{Larsen1998}). We further assume that $C$ is fixed and known by all customers and the system administrator.   

For an admission price $p$, a customer with service value $R$ that, observes a queue of length $q$ upon arrival, joins the queue if the expected utility is non-negative; $R-p-\frac{C(q+1)}{\mu}>0$. In other words, the customer joins the queue if the benefit $R-p$ is higher than the total expected time in the system multiplied by C, i.e.,~if $R-p \geq \frac{(q+1)C}{\mu}$. It is easy to show (see \cite{Naor1969}) that the customer balks if he observes $n_e = \lfloor\frac{(R-p) \mu}{C}\rfloor$ where $\lfloor x\rfloor$ is the floor function.  The probability of this event is
\(1-F_\theta(p + \frac{(q+1)C}{\mu})\). Thus, the joining process is not
Poisson with rate $\lambda$, but rather a Poisson process with state-dependent rates;
\[\lambda_q=\lambda \mathrm{P}\left( R \geq p + \frac{(q+1)C}{\mu} \right) = 
\lambda \left(1-F_\theta\left(p + \frac{(q+1)C}{\mu}\right)\right)\ , q=0,1,2,\ldots\ .\]
We denote the arrival probability at state $q$ by $\mathrm{P}_q=1-F\left(p + \frac{(q+1)C}{\mu}\right)$, and illustrate the transition diagram of the underlying CTMC in Figure~\ref{fig:transition}.

\begin{figure}
\centering
\includegraphics{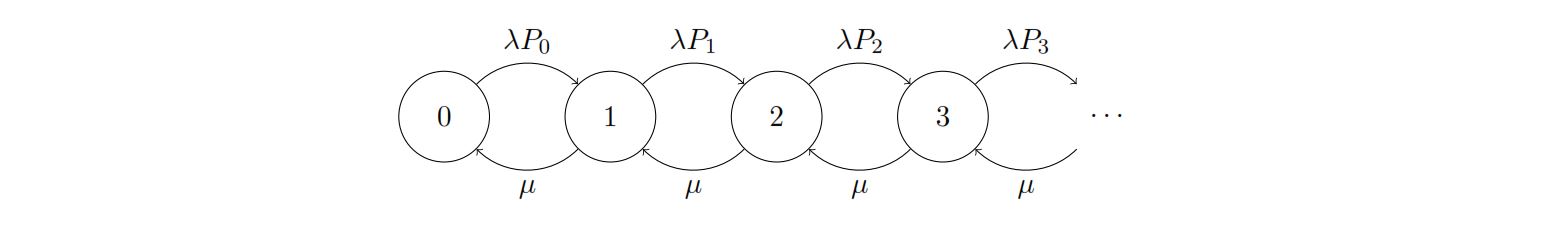}
\caption{Transition diagram of a queue with Poisson arrival process with
state dependent rate \(\lambda_q\) and exponential service rate
\(\mu\).}\label{fig:transition}
\end{figure}
\vspace{0.5cm}

The available data is the queue-length process: \((Q_0,\ldots,Q_k)\), where \(Q_0\) is an initial state and \(k\) is the sample size (or number of transitions).  This implies that the customers who balk are not observed; consequently, this information cannot be used.  This yields an interesting trade-off between revenue maximization and statistical inference. Namely, a system with a higher admission price will have a lower joining rate (at every system state), and so the data collection for inference purposes will be slower. Thus, the price can be reduced to learn faster and obtain the optimal price. However, too low price will considerably reduce the profit. In Naor's model, for a fixed reward, it is typically assumed that \(R\mu \geq C\), so that customers always join an empty queue. We replace this assumption with \(\mathrm{P}(R\geq C/\mu)=F_\theta(C/\mu)>0\). Therefore, when the queue is empty, there is a positive probability that a customer will join the queue. Otherwise, the queue is always empty and there is no information for inference.

In this work,  we assume that the distribution of $R$ is known and derive an Maximum Likelihood Estimator for its parameter \(\pmb{\theta}\) (possibly vector). There is no way to directly observe realizations of \(R\); however, the queue that can be observed contains information about it. Thus, assuming that the cost of staying in the system per unit of time \(C\) is constant, we estimate the parameters of \(R\) using queue information. Further we show that proposed estimator is well behaving, e.g. it is consistent with asymptotically normal errors. We also provide numerical examples with exponential and hyperexponential distributions. And finally we propose iterative pricing algorithm as a use case for this kind of estimators.

\subsection{Outline and contributions}\label{sec:outline-and-contribution}

\begin{itemize}

\item
In Section~\ref{sec:maximum-likelihood-estimator-for-pmbtheta.} we introduce the model and develop an MLE for the parameters of the service value distribution. 

\item
In Section~\ref{sec:asymptotic-properties.}  we show that under standard regularity assumptions,  the MLE is consistent and the normalized errors are asymptotically normally distributed. A formula for the asymptotic variance is further derived in terms of $F$ and the corresponding stationary distribution of the queue-length.

\item Section~\ref{sec:examples} provides details on the implementation of the MLE for two specific value distributions: exponential and hyperexponential. The conditions for the asymptotic guarantees are verified for these examples and simulation experiments illustrating the performance are analysed. Moreover, the asymptotic variance sheds insight on the performance of the MLE for various levels of the (total) arrival rate. Notably, it is hard to learn the true parameters when the arrival rate is very low (not enough customers join) or very high (the queue length remains high for long periods). The method works best for intermediate levels of the arrival rate because many customers join, while still allowing the queue to empty frequently. 

\item The usefulness of the estimation method and of the asymptotic results is demonstrated by leveraging them to construct a data-driven revenue maximization algorithm. In Section~\ref{sec:iterative-pricing-algorithm} we propose an iterative data-driven pricing algorithm. The algorithm sets a price at every iteration and collects queue length data for predefined duration (in terms of samples), and then updates the MLE. With the new MLE a new optimal price is estimated. This procedure is repeated with increasing sample sizes. The aforementioned asymptotic guarantees ensure that the algorithm converges to the revenue maximizing price. The performance of the algorithm is investigated through a series of simulation experiments. A notable observation is that in the case of hyperexponential service value, which can approximate a broad class of continuous distributions, the optimal price can be estimated accurately with a much smaller sample size than that required for an accurate MLE. In other words, knowledge of the exact parameters of the of the true service value distribution is not essential for revenue maximization. 

\item
Section~\ref{sec:conclusions-and-extensions.} provides concluding remarks and discusses several possible extensions of the method.

\end{itemize}

\subsection{Literature review}\label{sec:literature-review}

The first model that proposed quantitative analysis of an entrance fee (price) in queueing models with economic parameters was Naor's model \cite{Naor1969}. The model analyzed optimal pricing for profit and social welfare maximization for the FCFS \(M/M/1\) queueing model. In Naor's settings, the arriving customer observes the queue length and decides whether they want to join the queue or balk. The decision is based on \(C\) - the cost of staying in the system per unit of time, and the benefit gained from the completed service \(R\) and \(\mu\) - the service rate. As stated above, the individually optimal joining threshold is given by \(n_e = \lfloor\frac{(R-p) \mu}{C}\rfloor\). Naor further characterized the socially optimal threshold \(n^*\) and the monopolistic (profit-maximizing) threshold \(n_m\). In particular, he showed that \(n_m \leq n^* \leq n_e\), where the total social welfare is maximized under the assumption that individual benefits are additive. Customers can be `forced' to join according to the optimal thresholds by charging an admission fee (can be interpreted as service price) \(p\), or an increase of the cost per unit of time \(C\) (see \cite{book_HH2003}). The price can be decided on in different ways.

Naor's model was extended in \cite{Larsen1998} to allow for heterogeneous service values. In this case, \(R\) is assumed to be a continuous, non-negative random variable, and then the underlying process is an M/M/1 with state-dependent joining rates (see Figure~\ref{fig:transition}). \cite{Larsen1998} derived, numerically, the socially and monopoly optimal prices for \(R \sim U(a,b)\). It is further shown that the profit-maximizing fee is larger or equal to the fee that optimizes social welfare. However, this property does not necessarily hold if customers differ by their waiting time sensitivity (see \cite{book_HH2003}). Schroeter\footnote{The article is discussed in \cite{book_HH2003}, but it was never published. All attempts to find the original article were unsuccessful.} made a generalization in a different direction. He assumed \(C \sim U(0,b)\) and derived the profit-maximizing price. A customer observing a queue of length \(q_i\), upon arrival, joins the queue if \(C \leq \frac{(R-p)\mu}{q_i+1}\). Under some assumptions, Schroeter found that the profit-maximizing price admits a closed form solution. The question of maximizing social welfare and revenue for systems with heterogeneous customers has since been addressed in many studies with various model assumptions, and we refer the reader to \cite{book_H2016} for an in-depth review of these results.

There is a broad literature dealing with parameter and state estimation for queueing systems, which has recently been surveyed in \cite{asanjarani2021survey}. This literature focuses mostly on estimating the features of the queue, such as queue states, arrival and service rates, workload, idle times, and so on. In other words, the parameters of the queue itself. To the best of our knowledge, only a few works have concentrated on studying the customers' economic characteristics through the queuing process. One such work is \cite{ravner_estimating_2023}, which proposes an estimator for delay and tardiness sensitivities in a queue with strategic timing of arrivals. They show that the estimator is strongly consistent, and the normalized errors are asymptotically normal with a known covariance matrix. The estimator is constructed using the method of moments from the Wardrop equilibrium condition. There is a stream of literature that considers the problem of estimating customer patience for service systems with abandonments (e.g., \cite{AAES2013}). Note that the current work is also related to the classical statistical problem of asymptotic performance of maximum likelihood estimators of parameters of Markov chains and Birth-death processes in particular (e.g., \cite{B1961} and \cite{W2012}).

In the context of systems with unobserved balking, \cite{inoue2023estimating} deals with the maximum likelihood estimation of customer patience parameters based on effective inter-arrival time data. They assume customers have a random continuous patience level and will join the queue only if the delay is below this level. Note that this does not necessarily assume a utility-based decision as in Naor's model. Their method involves deriving a maximum likelihood estimator (MLE) based on observations of effective inter-arrival times. Strong consistency of the MLE and the asymptotic distribution of the estimation error, which is not always normal, are derived. This approach has been extended to a multi-server non-homogeneous in time arrival process in \cite{BMR2023}. A closely related problem is that of admission pricing in the presence of uncertainty regarding customer parameters. For example, in Naor's model, \cite{afeche_bayesian_2013} assume an identical reward and two types of customers: patient customers with a low waiting cost $c_L$, and impatient customers with a high waiting cost $c_H$. They propose a Bayesian dynamic pricing model (queue-length dependent) to estimate the proportion of patient customers and establish convergence to the optimal pricing policy. In \cite{CLH2023}, an online stochastic approximation algorithm is presented for the problem of jointly setting an admission price and capacity allocation. This is done for a G/G/1 queue where the incoming arrival rate is a function of the admission price, but not of the real-time congestion levels. Under mild assumptions the algorithm is shown to converge to the optimal policy. A final related work is \cite{RS2023} that presents a stochastic approximation algorithm that estimates the Nash equilibrium (rather than system optimal) strategy for a general class of queueing games.

The problem of estimating customer utility parameters in the presence of unobseved balking has recently been studied in \cite{BMR2023} and \cite{inoue2023estimating}. Similarly to the current paper, these works also exploit properties of the observable queueing model for MLE construction that aims to learn behavior of the customers. However, \cite{BMR2023} and \cite{inoue2023estimating} focus on the patience (delay sensitivity) of customers, while in the current paper we are interested in parameters of service value distribution (R). On a methodological level, the estimators in \cite{BMR2023} and \cite{inoue2023estimating} rely on the operator continuously observing the workload in the system and being able to convey this information accurately to arriving customers. The estimation procedure presented in this work relies only on discrete queue length data. Furthermore, customers make their joining decisions based on the queue length, which is also a more common assumption in service systems because exact real waiting time information is often unavailable.

\section{Maximum likelihood estimator for
\(\pmb{\theta}\)}\label{sec:maximum-likelihood-estimator-for-pmbtheta.}

Assume that customers arrive at a single server facility according to a
Poisson process with rate \(\lambda\). The service times are
independently, identically, and exponentially distributed with parameter
\(\mu\). Upon arrival, customer $i$ sees a queue length \(q_i\) and a fixed admission price \(p\). Assume that the customers are
homogeneous in their waiting time sensitivity,  i.e., \(C\) is constant, and that this value is known. Also, \(R\) is a RV with some CDF \(F(r,\pmb{\theta})\).
For the sake of brevity, we simplify \(F(r,\pmb{\theta})\) to \(F(r)\),
however, the reader should keep in mind that \(F\) does depend on
\(\pmb{\theta}\). In addition, as was previously mentioned, we assume that there is a positive probability that a customer joins an empty queue; \(F(C/\mu,\pmb{\theta})>0\).

The queue-length process is clearly a continuous-time Markov chain (specifically, a birth-death process) on the non-negative integers, with state dependent transition rates. Throughout this work we consider the underlying discrete-time jump chain. We refer to a transition in the queue state as a ``step'', and the
initial state of the system is the queue length in step \(0\) - before
the first jump. Thus, the system state at step \(i\) is the queue length
after \(i\) jumps.  Let $Q_i$ denote the state at step $i\geq 0$,. Given \(Q_{i-1}=q_{i-1}\)), applying
the Markov property, the probability that in step \(i\) the queue length
(\(Q_{i}\)) increases by one is
\[
\mathrm{P}(Q_{i}=q_{i-1}+1 | Q_{i-1}=q_{i-1})=\left\{\begin{matrix}
1 &,\:\: q_{i-1}=0  \\
\frac{\lambda_i}{\lambda_i+\mu}  &,\:\: q_{i-1}>0
\end{matrix}\right.\;,
\]
where
\[
\lambda_i:=\lambda\left(1-F\left(p+\frac{C(q_{i-1}+1)}{\mu}\right)\right)\;\;.
\]
Similarly,
\[
\mathrm{P}(Q_{i}=q_{i-1}-1 | Q_{i-1}=q_{i-1})=\left\{\begin{matrix}
0 &,\:\: q_{i-1}=0  \\
\frac{\mu}{\lambda_i+\mu}  &,\:\: q_{i-1}>0
\end{matrix}\right.\;.
\]
For the sake of a compact representation,  let \(r(q_i)=p+\frac{(q_i+1)C}{\mu}\).  Iterating the Markov property,  the probability to see a path
\(q_k,q_{k-1},q_{k-2},...,q_{0}\), for a given parameter
\(\pmb{\theta}\), is
\begin{align*}
& \mathrm{P}(Q_k=q_k,Q_{k-1}=q_{k-1},Q_{k-2}=q_{k-2},...,Q_{0}=q_{0},\pmb{\theta})=\prod_{i=1}^k\mathrm{P}(Q_{i}=q_{i}|Q_{i-1}=q_{i-1},\pmb{\theta}) \\
& \ =\prod_{i\in K_{1}} \mathrm{P}(Q_{i}=q_{i-1}+1 | Q_{i-1}=q_{i-1},\pmb{\theta}) \prod_{i\in K_2} \mathrm{P}(Q_{i}=q_{i-1}-1 | Q_{i-1}=q_{i-1},\pmb{\theta})\prod_{i\in K_{3}} 1, 
\end{align*}
where \(K_1=\{1\leq i\leq k\:|\: Q_{i}=q_{i-1}+1, q_{i-1}\neq0 \}\),
\(K_2=\{1\leq i\leq k\:|\: Q_{i}=q_{i-1}-1,\: F(r(q_{i-1}))\neq1 \}\)
and
\(K_3=\{1\leq i\leq k\:|\: F(r(q_{i-1}))\in\left\{ 0,1 \right\} \}\).
\(K_3\) is a group of indices where the probability of upward or
downward transition is \(1\).
\(\mathrm{P}(Q_{i}=q_{i-1}+1 | Q_{i-1}=q_{i-1})=1\) when
\(F(r(q_{i-1}))=0\) and
\(\mathrm{P}(Q_{i}=q_{i-1}-1 | Q_{i-1}=q_{i-1})=1\) when
\(F(r(q_{i-1}))=1\). Thus, this likelihood can be rewritten as
\[
\mathcal{L}(\pmb{\theta}|Q_k):=
\prod_{i\in K_{1}} \frac{\lambda(1-F(r(q_{i-1}),\pmb{\theta}))}{\mu+\lambda(1-F(r(q_{i-1}),\pmb{\theta}))}
\prod_{i\in K_2} \frac{\mu}{\mu+\lambda(1-F(r(q_{i-1}),\pmb{\theta}))}\prod_{i\in K_{3}}1 \;\;.
\]

Thus, after \(k\) steps, the log-likelihood of \(\pmb{\theta}\) is:
\begin{align*}
 \log \mathcal{L}(\pmb{\theta}|Q_k) &=
\log \left\{\Pi_{i\in K_{1}} \frac{\lambda(1-F(r(q_{i-1}),\pmb{\theta}))}{\mu+\lambda(1-F(r(q_{i-1}),\pmb{\theta}))}
\Pi_{i\in K_2} \frac{\mu}{\mu+\lambda(1-F(r(q_{i-1}),\pmb{\theta}))}\Pi_{i\in K_{3}}1\right\} \\
&\ =-\sum_{i\in K_1\cup K_2}\log\left\{\mu+\lambda\left(1-F(r(q_{i-1}),\pmb{\theta})\right) \right\}+
\sum_{i\in K_1}\log\left\{ \lambda(1-F\left(r(q_{i-1}),\pmb{\theta})\right) \right\}+
|K_2|\log\left\{ \mu \right\}\;,
\end{align*}
where \(|K_2|\) is the cardinality of \(K_2\).

The Maximum Likelihood Estimator (MLE) is the solution of the
optimization problem
\[
\hat{\pmb{\theta}}_k=\underset{\pmb{\theta}\in\pmb{\Theta}}{\text{argmax}}\left\{ \log \mathcal{L}(\pmb{\theta}|Q_k) \right\},
\]
where \(\pmb{\Theta}\subset\mathbb{R}^n\) is the parameter space. If the
solution is interior, it can be derived with the first-order condition
\begin{align*}\label{esteq}
\nabla  \log \mathcal{L}(\pmb{\theta}|Q_k) =\pmb{0},
\end{align*}
where \(\nabla  \log \mathcal{L}(\pmb{\theta}|Q_k)\) is a vector-valued gradient
\(\mathbb{R}^n \rightarrow\mathbb{R}^n\) of log-likelihood function. Or,
in more detail,
\[
\begin{matrix}
\frac{\partial\log \mathcal{L}(\theta_j|Q_k)}{\partial \theta_j}=
\sum_{i\in K_1\cup K_2}\left(  \frac{\lambda F_{\theta_j}'(r(q_{i-1}),\pmb{\theta})}{\mu+\lambda(1-F(r(q_{i-1}),\pmb{\theta}))}  \right)-
\sum_{i\in K_1}\left( \frac{\lambda F_{\theta_j}'(r(q_{i-1}),\pmb{\theta})}{\lambda(1-F(r(q_{i-1}),\pmb{\theta}))} \right)=0\;,\\
\\
\forall j=1.\ldots, n
\end{matrix}
\]
where \(F_{\theta_j}'=\frac{\partial F}{\partial \theta_j}\). Otherwise,
it lies on the boundary of \(\pmb{\Theta}\). In addition to \(Q_k\) as
defined previously, we denote the random variable
\(\mathbf{Y}_k=(Y_1,Y_2,...,Y_k)\) where \(Y_i\) is indicator of an
upward transition \(Y_i=1(Q_i>Q_{i-1})\). Thus, conditional on
\(Q_{i-1}=q_{i-1}\),
\[
Y_i \sim Ber(p(q_{i-1},\pmb{\theta})),
\]
where $p(q_{i-1},\pmb{\theta}):=\mathrm{P}(Q_{i}=q_{i-1}+1 | Q_{i-1}=q_{i-1})$. In addition, we introduce \(M:=K_1\cup K_2=K \backslash K_3\). The set
\(M\) of indexes \(i\) , is the set of indexes where
\(F(r(q_{i-1}),\pmb{\theta})\notin\left\{ 0,1 \right\}\), formally:
\(M=\{1\leq i\leq k\:|\: F(r(q_{i-1}),\pmb{\theta})\notin\left\{ 0,1 \right\} \}\).
\(M\) can be thought as effective sample and \(|M|\) as effective sample
size, while \(K\) is the full sample and \(|K|=k\) is the full sample
size.

For any \(q\) such that
\(F(r(q_{i-1}),\pmb{\theta})\in\left\{ 0,1 \right\}\),
\(p(q_{i-1},\pmb{\theta})\in \left\{ 1,0 \right\}\), so
\(\frac{\partial p(0,\theta_j)}{\partial \theta_j}=0\),
\(\forall j \in n\). This case does not add any information to the
likelihood function or its derivative. Thus, to simplify further
notions, we ignore this case and consider only the sample of observations with indices in $M$.  It is important to note that it
does not affect assumptions and results. Observe that \(p(q_{i-1},\theta)=\frac{\lambda_i}{\lambda_i+\mu}\) for any \(i\in M\).
Then,  the log-likelihood can be rewritten as
\begin{align*}\label{logder}
\log \mathcal{L}(\pmb{\theta}|Q_k,Y_k)=
\frac{1}{k}\sum_{i\in M}[Y_i \log\left\{p(Q_{i-1},\pmb{\theta})\right\}+(1-Y_i)\log\left\{1-p(Q_{i-1},\pmb{\theta})\right\}].
\end{align*}

The gradient is given by
\[\nabla\log \mathcal{L}(\pmb{\theta}|Q_k,Y_k):=\pmb{\Psi}_k(\pmb{\theta})=\frac{1}{k}\sum_{i\in M}\pmb{\psi}(Q_{i-1},Y_i,\pmb{\theta})\]
where
\begin{align*}\label{psi_i_def}
\pmb{\psi}(Q_{i-1},Y_i,\pmb{\theta})=
\begin{bmatrix}
\psi^1(Q_{i-1},Y_i,\pmb{\theta})
\\
\\
\vdots
\\
\\
\psi^n(Q_{i-1},Y_i,\pmb{\theta})
\end{bmatrix}=
\begin{bmatrix}
Y_i \frac{p'_{\theta_1}(Q_{i-1},\pmb{\theta})}{p(Q_{i-1},\pmb{\theta})}
- (1-Y_i) \frac{p'_{\theta_1}(Q_{i-1},\pmb{\theta})}{1-p(Q_{i-1},\pmb{\theta})}
\\
\\
\vdots
\\
\\
Y_i \frac{p'_{\theta_n}(Q_{i-1},\pmb{\theta})}{p(Q_{i-1},\pmb{\theta})}
- (1-Y_i) \frac{p'_{\theta_n}(Q_{i-1},\pmb{\theta})}{1-p(Q_{i-1},\pmb{\theta})}
\end{bmatrix}
\end{align*}

\begin{remark}
  Frequently, observed information can be noisy and this can affect the model performance and theoretical results. for instance, the noise can be in queue length observed by arriving customer. Alternatively, the noise can appear in the data collected by administrator. We briefly address how our analysis can be modified to account for such setting.
  
Suppose that upon arrival of customer $i$, the queue length is $q_i$, however there is an error in observation, and the customer observes queue of length $u_i$.  The probability to see the queue of length $u_i$ when the actual queue is $q_i$, is $P(U_i=u_i|Q_i=q_i)$. Then the probability that customer will join when the actual queue length upon arrival is $q_i$ is
$$\sum_{u_i=0}^{\infty}\left(1-F\left(r(u_i)\right)\right)P(U_i=u_i|Q_i=q_i)$$
which replaces the probability to join in system with clean information $1-F\left(r(q_i)\right)$. If the conditional distribution of $U_i$ is known, it can be taken into account as was shown. In this case the assumptions verification may be challenging, however the general estimation approach remains unchanged.

The other case is noisy data collection. The effect depends on how data is collected. For example, one can consider a case when system state is computed from registered arrivals and departures. Let $\{I^a|Q=q\}$ be an indicator that an arrival is registered given the system size $q$. Similarly, let $\{I^d|Q=q\}$ be an indicator that an departure is registered given the system size $q$. Then $\{I^a|Q=q\}\sim Ber(p_q^a)$ and $\{I^d|Q=q\}\sim Ber(p_q^d)$. Then, each upward probability and downward probability in the likelihood function must be multiplied by $p_q^a$ and $p_q^d$ correspondingly. Note that now the likelihood depends also on the parameters $(p_q^a,p_q^d)$ which can be jointly estimated together with $\theta$. 
\end{remark}

\section{Asymptotic properties.}\label{sec:asymptotic-properties.}

In current section we present sufficient conditions for consistency of the MLE, and for asymptotic normality of the estimation errors. Consistency provide theoretical guarantees that estimated value converges to the true one as sample size grows to infinity. In practice one can be sure that good estimation is only a matter of amount of data and computational power. Knowledge of asymptotic distribution of the estimation error can be very useful as provides considerably more information than only estimation. It can be used for proximity evaluation of single estimation from the theoretical value. One can use it for statistical testing and computation of confidence interval.

The fact that customers are less likely to join as the queue grows,  i.e., $1-F(r)\to0$ as $r\to\infty$,  ensures that the queue-length process is stable. Formally, we will later assert that the sequence $(Q_{i-1},Y_i)$ converges to a stationary distribution. We will assume throughout that the initial state follows the stationary distribution, hence the stationary random vector is represented by $(Q_0,Y_1)$. The ergodic limit of the gradient with respect to the parameter, if it exists, is denoted by $\pmb{\Psi}$; 
\begin{align*}
\pmb{\Psi}_k(\pmb{\theta})\overset{a.s.}{\longrightarrow}\pmb{\Psi}(\pmb{\theta})\ , \pmb{\theta}\in\Theta. 
\end{align*}

Note that we use $||\cdot||$ for the euclidean distance of vectors in $\mathbb{R}^n$, and $|\cdot|$ for a vector of absolute values of their coordinates.  For a matrix $\mathbf{A}$, its transpose is denoted by $\mathbf{A}^T$ and its inverse by $\mathbf{A}^{-1}$.  Let $\pmb{\theta}_0$ be true value of $\pmb{\theta}$, i.e. actual decision of arriving customers is based on $R\sim F(r,\pmb{\theta}_0)$. We further assume that $F$ is identifiable; there exist no $\pmb{\theta}\neq \pmb{\theta}_0$ such that $F(r,\pmb{\theta})=F(r,\pmb{\theta}_0)$ for all $r\geq 0$.

\noindent{\bf Assumptions}

\newcommand{\AssOnTheta}{A1}
\textbf{A1.} \(\pmb{\Theta}\) is a compact and convex set and
\(\pmb{\theta}_0\) is an interior point in \(\pmb{\Theta}\).

\newcommand{\AssContDiffInTheta}{A2}
\textbf{A2.} \(F(r,\pmb{\theta})\) is continuous w.r.t.
\(\pmb{\theta}\in\pmb{\Theta}\). Furthermore,  the gradient $\nabla F$ is also continuous and differentiable w.r.t. \(\pmb{\theta}\in\pmb{\Theta}\).  

\newcommand{\AssPsiBoundness}{A3}
\textbf{A3.}   There exists a real valued function $H:\mathbb{N}\to\mathbb{R}$ such that
\begin{align*}
\max_{j=1,\ldots,n}\left|\frac{ \nabla F(r(q),\pmb{\theta})}{1-F(r(q),\pmb{\theta})}\right|_j\leq H(q)\ , \forall q\geq 0 ,
\end{align*}
and $\mathbb{E}[H(Q_0)]<\infty$.

\newcommand{\AssHessianBound}{A4}
\textbf{A4.} For any $j,l\in\{1,\ldots,n\}$, let
\begin{align*}
\Sigma_{jl}:=\mathbb{E}\left[ \frac{\mu\lambda F'_{\theta_j}(r(Q_0),\pmb{\theta}_0)F'_{\theta_l}(r(Q_0),\pmb{\theta}_0)}{(1-F(r(Q_0),\pmb{\theta}_0))(\mu+\lambda(1-F(r(Q_0),\pmb{\theta}_0)))^2} \right].
\end{align*}
 The matrix $\pmb{\Sigma}$ is invertible and $|\Sigma_{jl}|<\infty$ for any $j,l\in\{1,\ldots,n\}$.

Our asymptotic results rely on the framework of \cite{book_vdV1998} and the assumptions are standard in statistical theory. The main modification is that we are dealing with indirect estimation. In particular, the data is not an iid sample of service value observations, but rather the Markov chain $(Q_{i-1},Y_i)$ whose transition matrix is determined by the service value distribution.  
For consistency the main idea is to ensure that $\pmb{\psi}(Q_{i-1},Y_i,\pmb{\theta})$ converges uniformly to a ``well-behaved'' theoretical score function that has a unique root at the true parameter value. Assumptions A1,  and A2 are used to establish continuity of \(\pmb{\psi}(q,y,\pmb{\theta})\) and its first two
derivatives w.r.t.  \(\pmb{\theta}\). In addition, in order
to show that MLE is consistent we must assume that \(\pmb{\theta}_0\) is an interior point of the parameter space (A1). If
\(\pmb{\theta}_0\) lies on the boundary of \(\pmb{\Theta}\) the
convergence to the true value is not guaranteed and the limiting
distribution of the error may not be normal. Assumption A3 will be used to construct a uniform bound on $\pmb{\psi}$ that is integrable with respect to the stationary distribution, which together with A1, A2 implies uniform convergence of $\pmb{\Psi}_k(\pmb{\theta})$ on $\pmb{\Theta}$. The last assumption (A4) is used to show that the stationary variance of the estimation error exists and is finite. This will be used to verify a martingale CLT implying that the estimation errors are normally distributed.

Perhaps the only non-standard assumptions here that arise from the special structure of the process is A3 and A4. To gain some additional intuition for these condition, \(1-F(r(q),\pmb{\theta})\) can be thought as balking rate, while \(\nabla_{\pmb{\theta}} F(r(q),\pmb{\theta})\) is gradient of
\(F(r(q),\pmb{\theta})\) w.r.t. \(\pmb{\theta}\). It converges to the
true value as number of observations grows. However, if the balking rate grows too fast as \(q\rightarrow \infty\), the effective sample can be too small.  Note, however, that this is a sufficient but not  necessary condition.  In Section~\ref{sec:examples} these conditions are verified for the special cases of exponentially and hyperexponentially distributed $R$.

\subsection{Consistency.}\label{sec:consistency.}

\begin{theorem}\label{theorem1}
Under Assumptions A1, A2, A3, the MLE is consistent; $\hat{\pmb{\theta}}_k\overset{p}{\rightarrow}\pmb{\theta}_0$ as $k\rightarrow\infty$.
\end{theorem}

\begin{proof}

To show consistency we verify the sufficient conditions given in \cite[Thm.~5.9]{book_vdV1998}.  In particular, we need to show that 
\begin{equation}\label{eq:unique_theta0}
\inf_{\pmb{\theta}:d(\pmb{\theta},\pmb{\theta}_0)\geq \epsilon}||\pmb{\Psi}(\pmb{\theta})||>0=||\pmb{\Psi}(\pmb{\theta}_0)||,
\end{equation}
and that the class of functions $\{\pmb{\Psi}_k(\pmb{\theta}): \ \pmb{\theta}\in\pmb{\Theta}\}$ is Glinvenko-Cantelli (in the strong sense);
\begin{equation}\label{eq:uniform}
\sup_{\pmb{\theta}\in\pmb{\Theta}}\left|\left|\pmb{\Psi}_k(\pmb{\pmb{\theta}})-\pmb{\Psi}(\pmb{\pmb{\theta}})\right|\right| \overset{\mathrm{a.s.}}{\longrightarrow} 0.
\end{equation}
Then, by \cite[Thm.~5.9]{book_vdV1998}, any sequence of estimators such that 
\begin{equation}\label{eq:root}
\pmb{\Psi}_k(\hat{\pmb{\theta}}_k)\overset{P}{\longrightarrow}\pmb{0},
\end{equation}
 is consistent. The proof relies on a sequence of lemmas.  First of all, Lemma~\ref{lem:continuity} establishes the continuity of $\pmb{\psi}$ under Assumption~A2.  Lemma~\ref{lem:bound} shows that $\pmb{\psi}$ is bounded by the integrable function $H$ in Assumption~A3 (up to a multiplicative constant).  Note that we are only interested in values of $q$ such that $F(r(q),\pmb{\theta})\in(0,1)$ because our effective observation set $M$ only consists of such values.  Lemma~\ref{lem:ergodic} shows that the sequence $(\pmb{\Psi}_k(\pmb{\pmb{\theta}}))_{k\geq 1}$ has an ergodic limit for any $\pmb{\theta}\in\pmb{\Theta}$. Lemma~\ref{lem:unique} builds on the previous results to verify \eqref{eq:unique_theta0}. By Assumption~A1, $\pmb{\Theta}$ is compact, and Lemmas~\ref{lem:continuity},\ref{lem:bound} and \ref{lem:ergodic} verify that the sequence of functions $\{\pmb{\Psi}_k(\pmb{\theta}): \ \pmb{\theta}\in\pmb{\Theta}\}$ is continuous, bounded by an integrable function and pointwise converegent, thus \eqref{eq:uniform} holds (e.g.,  \cite[Thm.~16a]{book_F1996}). Finally, Lemma~\ref{lem:root} verifies \eqref{eq:root}. The proofs of all of the lemmas are provided in Appendix~\ref{appendix-1.-proof-of-lemmas-1-5.}.

\begin{lemma}\label{lem:continuity}
If $F(r(q),\pmb{\theta})$ is continuous, twice differentiable and has a continuous first derivative, w.r.t. $\pmb{\theta}\in \pmb{\Theta}$ (Assumption~A2), then  $\pmb{\psi}_i(q,y,\pmb{\theta})$ is continuous  w.r.t. $\pmb{\theta}\in \pmb{\Theta}$, for any $(q,y)$ such that $F(r(q),\pmb{\theta})\in(0,1)$.
\end{lemma}

\begin{lemma}\label{lem:bound}
Suppose Assumption A3 holds, then
\begin{align*}
\max_{j=1\ldots,n}\left|\pmb{\Psi}_k(\pmb{\pmb{\theta}})\right|_j\leq 2 H(q)\ , \forall q: F(r(q),\pmb{\theta})\in(0,1),
\end{align*}
 where $H$ is the integrable function in Assumption~A3.
\end{lemma}

\begin{lemma}\label{lem:ergodic}
A unique stationary distribution $(Q_0,Y_1)$ exists and for any integrable function $g$ such that $\mathbb{E}[|g(Q_0,Y_1)|]<\infty$,
$$\frac{1}{k}\sum_{i\in M} g(Q_{i-1},Y_{i})\overset{a.s.}{\longrightarrow}\mathbb{E}[g(Q_0,Y_1)]\:\: \text{as} \:\:k\rightarrow\infty.$$
If Assumption A3\: holds then in particular
$$\pmb{\Psi}_k(\pmb{\theta})\overset{a.s.}{\longrightarrow}\pmb{\Psi}(\pmb{\theta})\:\: \text{as }\: k\rightarrow\infty.$$
\end{lemma}

\begin{lemma}\label{lem:unique}
If Assumption A3 holds, then for any $\epsilon>0$.
$$\inf_{\pmb{\theta}:d(\pmb{\theta},\pmb{\theta}_0)\geq \epsilon}||\pmb{\Psi}(\pmb{\theta})||>0=||\pmb{\Psi}(\pmb{\theta}_0)||.$$
\end{lemma}

\begin{lemma}\label{lem:root}
If Assumptions A1, A2, A3 hold, then
$$\pmb{\Psi}_k(\hat{\pmb{\theta}}_k)\overset{a.s.}{\longrightarrow}\pmb{0}\:$$
and consequently
$$\pmb{\Psi}_k(\hat{\pmb{\theta}}_k)\overset{P}{\longrightarrow}\pmb{0}\:.$$
\end{lemma}

This completes the proof of Theorem~\ref{theorem1}. 
\end{proof}

\subsection{Asymptotic normality.}\label{sec:asymptotic-normality.}

We next focus on the asymptotic distribution of the estimation error. Specifically, we will show that under Assumptions A1-A4, the $\sqrt{k}$-scaled estimation errors converge in distribution to a normally distributed random variable with a computable covariance matrix.

\begin{theorem}\label{theorem2}
If Assumptions A1-A4\: hold and $Q_0$ is stationary, then
$$\sqrt{k}(\hat{\pmb{\theta}}_k-\pmb{\theta}_0)\overset{d}{\longrightarrow}N_n\left(\pmb{0},-\pmb{\nabla}\pmb{\Psi}^{-1}(\pmb{\theta}_0)\right),$$
where
\begin{align*}
(\pmb{\nabla}\pmb{\Psi}(\pmb{\theta}_0))_{jl}=-\Sigma_{jl}=\mathbb{E}\left[ \frac{\mu\lambda F'_{\theta_j}(r(Q_0),\pmb{\theta}_0)F'_{\theta_l}(r(Q_0),\pmb{\theta}_0)}{(1-F(r(Q_0),\pmb{\theta}_0))(\mu+\lambda(1-F(r(Q_0),\pmb{\theta}_0)))^2} \right] , \ j,l\in\{1,\ldots,n\}.
\end{align*}
\end{theorem}

\begin{proof} By the convexity of $\Theta$ in Assumption ~A1, for any interior point \(\hat{\pmb{\theta}}_k\in\pmb{\Theta}^\mathrm{o}\), by the Mean Value Theorem there exists a value \(t\in [0, 1]\) such that
\[
\pmb{\Psi}_k(\hat{\pmb{\theta}}_k)=\pmb{\Psi}_k(\pmb{\theta}_0)+\pmb{\nabla}\pmb{\Psi}_k(\pmb{\theta}_0+t(\hat{\pmb{\theta}}_k-\pmb{\theta}_0))(\hat{\pmb{\theta}}_k-\pmb{\theta}_0),
\]
where \(\pmb{\nabla}\pmb{\Psi}_k\) is the Jacobian of $\pmb{\Psi}_k(\pmb{\theta})$. By Theorem~\ref{theorem1},
\(\hat{\pmb{\theta}}_k\overset{p}{\rightarrow}\pmb{\theta}_0\). As a
result
\(t(\hat{\pmb{\theta}}_k-\pmb{\theta}_0)k\overset{p}{\rightarrow}0\)
which also implies convergence in distribution. In addition, by Lemma
\ref{lem:root},  there exists some large $K$ such that \(\pmb{\Psi}_k(\hat{\pmb{\theta}}_k)=\pmb{0}\) for all \(k>K\), almost surely.  Hence , for \(k>K\) 
\[
\pmb{0}=\pmb{\Psi}_k(\pmb{\theta}_0)+\pmb{\nabla}\pmb{\Psi}_k(\pmb{\theta}_0)(\hat{\pmb{\theta}}_k-\pmb{\theta}_0),
\]
which implies that the right hand-side converges to a vector of zeroes in distribution.  For now we assume that \(\pmb{\nabla}\pmb{\Psi}_k(\pmb{\theta}_0)\) is invertible (later this will shown to be true under Assumption~A4), yielding
\[
\sqrt{k}(\hat{\pmb{\theta}}_k-\pmb{\theta}_0)=-\sqrt{k}\pmb{\nabla}\pmb{\Psi}_k^{-1}(\pmb{\theta}_0)\pmb{\Psi}_k(\pmb{\theta}_0),
\]
Now, suppose that for some covariance matrix $\pmb{\Sigma}$,
\begin{align}\label{eq:normCond1}
\sqrt{k}\pmb{\Psi}_k(\pmb{\theta}_0)\overset{d}{\longrightarrow}N_n(\pmb{0},\pmb{\Sigma}),
\end{align}
and
\begin{align}\label{eq:normCond2}
\pmb{\nabla}\pmb{\Psi}_k^{-1}(\pmb{\theta}_0)\overset{\mathrm{p}}{\longrightarrow}\pmb{\nabla}\pmb{\Psi}^{-1}(\pmb{\theta}_0),
\end{align}
then by Sluzky's Theorem
\[
\sqrt{k}(\hat{\pmb{\theta}}_k-\pmb{\theta}_0)\overset{d}{\longrightarrow}-\pmb{\nabla}\pmb{\Psi}^{-1}(\pmb{\theta}_0)N_n(\pmb{0},\pmb{\Sigma})=N_n(\pmb{0},\pmb{\nabla}\pmb{\Psi}^{-1}(\pmb{\theta}_0)\pmb{\Sigma}(\pmb{\nabla}\pmb{\Psi}^{-1}(\pmb{\theta}_0))^T)
\]
Lemma~\ref{lem:var} verifies \eqref{eq:normCond2} and that \(\pmb{\nabla}\pmb{\Psi}(\pmb{\theta}_0)=-\pmb{\Sigma}\), where $\pmb{\Sigma}$ is as defined in Assumption~A4.  Then, Lemma~\ref{lem:normal} establishes \eqref{eq:normCond1} by applying a martingale CLT.   Combining these results we have that
\[
\pmb{\nabla}\pmb{\Psi}^{-1}(\pmb{\theta}_0)\pmb{\Sigma}(\pmb{\nabla}\pmb{\Psi}^{-1}(\pmb{\theta}_0))^T=
\pmb{\nabla}\pmb{\Psi}^{-1}(\pmb{\theta}_0)\pmb{\nabla}\pmb{\Psi}(\pmb{\theta}_0)(\pmb{\nabla}\pmb{\Psi}^{-1}(\pmb{\theta}_0))^T=(\pmb{\nabla}\pmb{\Psi}^{-1}(\pmb{\theta}_0))^T,
\]
and using the fact that \(\pmb{\nabla}\pmb{\Psi}(\pmb{\theta}_0)\) is symmetric, since it is a Hessian matrix, we  conclude that
\[
(\pmb{\nabla}\pmb{\Psi}^{-1}(\pmb{\theta}_0))^T=\pmb{\nabla}\pmb{\Psi}^{-1}(\pmb{\theta}_0),
\]
 thus completing the proof of Theorem~\ref{theorem2}.  

\end{proof}

\begin{lemma}\label{lem:var}
If Assumptions  A1-A4 hold, then
\begin{align*}
\pmb{\nabla}\pmb{\Psi}_k^{-1}(\pmb{\theta}_0)\overset{\mathrm{a.s.}}{\longrightarrow}\pmb{\nabla}\pmb{\Psi}^{-1}(\pmb{\theta}_0)\ ,
\end{align*}
where
\begin{align*}
-\pmb{\nabla}\pmb{\Psi}(\pmb{\theta}_0)=\mathbb{E}[\pmb{\psi}(Q_0,Y_1,\pmb{\theta}_0)\pmb{\psi}(Q_0,Y_1,\pmb{\theta}_0)^T]=\pmb{\Sigma}.
\end{align*}
\end{lemma}
The proof of Lemma~\ref{lem:var} is given in Appendix~\ref{appendix-normal}.

\begin{lemma}\label{lem:normal}
If Assumptions  A1-A4 hold, then
$$\sqrt{k}\pmb{\Psi}_k(\pmb{\theta}_0)\overset{d}{\longrightarrow}N_n(\pmb{0},\pmb{\Sigma})$$
\end{lemma}

\begin{proof}
To prove Lemma \ref{lem:normal} we apply a martingale CLT for stationary ergodic sequences (e.g., \cite[Corr.~4.17]{vdV2010}).  Note that the result is typically stated for one dimensional sequences, but the extension to $\mathbb{R}^n$ is immediate by the Cramer-Wold device (see \cite[Ch.~2]{book_vdV1998}). Specifically, if \(\{\pmb{\psi}_{i}\}_{i=0}^{k-1}\) is a stationary, ergodic, martingale difference with respect to the filtration
\(\{\mathcal{F}_{i}\}_{i=0}^{k-1}\) such that
\(\mathbb{E}\left[\pmb{\psi}_0\pmb{\psi}_0^T\right]=\pmb{\Sigma}<\infty\)
and \(\mathbb{E}(\pmb{\psi}_{i}|\mathcal{F}_{i-1})=\pmb{0}\). Then
\[
\sqrt{k}\pmb{\Psi}_k\overset{d}{\longrightarrow}N(\pmb{0},\pmb{\Sigma}),
\]
where \(\sqrt{k}\pmb{\Psi}_k=\frac{1}{\sqrt{k}}\sum_{i=1}^k\pmb{\psi}_{i}\). We verify that the sequence $(\pmb{\Psi}_k(\pmb{\theta}_0))_{k\geq q}$ satisfies these conditions in the following four steps.

\noindent\textbf{Part 1: Filtration and
adaptation}\label{part-1-filtration-and-adaptation}
We assert that \(\pmb{\psi}_i\) is adapted (i.e.~measurable
w.r.t.) to the increasing filtration
\(\{\mathcal{F}_{i}\}_{i=0}^{n-1}\). As follows from \cite[Thm. 2.1.5]{bogachev_measure_2007}), if the function \(f\) is measurable with respect to a \(\sigma\)-algebra \(\mathcal{A}\) then the function \(\varphi \circ f\)
is measurable with respect to \(\mathcal{A}\) for any Borel function
\(\varphi:\mathbb{R}^1\rightarrow \mathbb{R}^1\). In Appendix~\ref{appendix-filtration} a construction of the sequence
\((Q_i)_{0\leq i\leq k}:(\Omega_k,\mathcal{F}_i,\mathrm{P}_i)\rightarrow(\mathbb{Z}^+,\{0,1,...,i\},F_i)\) is provided, and it is further argued that \(\pmb{\psi}_i(Q_{i-1},Y_i,\pmb{\theta})\) is measurable w.r.t.
\(\mathcal{F}_i\) for any $\pmb{\theta}\in\pmb{\Theta}$.

\noindent\textbf{Part 2: Ergodicity}\label{part-2-ergodicity}

Stationarity of \((Q_{i-1},Y_{i})\) implies stationarity of
\(\pmb{\psi}_i(Q_{i-1},Y_i,\pmb{\theta})\) as it is a measurable function of
\((Q_{i-1},Y_{i})\).  Thus,  by Lemma \ref{lem:ergodic} the sequence is also
ergodic.

\noindent\textbf{Part 3: Variance}\label{part-3-variance.}
In Lemma~\ref{lem:var} it was shown that the asymptotic variance of $\sqrt{k}\pmb{\Psi}_k(\pmb{\theta})$ is $\pmb{\Sigma}$, and by Assumption~A4 this is an invertible covariance matrix with finite entries.

\noindent\textbf{Part 4: Martingale}\label{part-4-martingale}
In the proof of Lemma \ref{lem:unique} (Appendix~\ref{appendix-1.-proof-of-lemmas-1-5.}) it is verified that
\[
\mathbb{E}(\psi_i^j(Q_{i-1},Y_i,\pmb{\theta}_0)|\mathcal{F}_{i-1})=\mathbb{E}\left[\left[ Y_i \frac{p'_{\theta_i}(Q_{i-1},\pmb{\theta}_0)}{p(Q_{i-1},\pmb{\theta})}
- (1-Y_i) \frac{p'_{\theta_i}(Q_{i-1},\pmb{\theta}_0)}{1-p(Q_{i-1},\pmb{\theta}_0)}\right]|Q_{i-1}\right]=0 ,
\]
where it is assumed that $Q_i\sim Q_0$ for all $i\geq 0$. Therefore, $\pmb{\Psi}_k(\pmb{\theta}_0)$ is indeed a martingale. 
\end{proof}

\section{Implementation}\label{sec:examples}

In this section we implement the MLE for two examples: exponential and hyperexponential service value distributions. In both cases the assumptions for the asymptotic guarantees are verified and numerical analysis is carried out on simulated data.

\subsection{Exponential service value.}\label{sec:exponential-service-value.}

Suppose that the service values are exponentially distributed; \(R\sim \exp(\theta)\) for some $\theta>0$.  Then,
\begin{align}\label{eq:exp_F}
F(r,\theta)=1-\exp(-\theta r)\ ,  r\geq 0 .
\end{align}
We will first verify that all of the assumptions required for our asymptotic results hold in this case. This will be followed by analysis of simulation experiments involving the estimation.

\textbf{Verification of
assumptions}\label{verification-of-assumptions}

Taking derivative of \eqref{eq:exp_F} yields
\begin{equation}\label{eq:exp_F_der_theta}
F_{\theta}'(r(q_i),\theta)=r(q_i)\exp(-\theta r(q_i)).
\end{equation}

The log-likelihood is then
\begin{align*}
    \log \mathcal{L}(\theta|Q_k)&= -\sum_{i\in K_1\cup K_2}\log\left\{\mu+\lambda\exp\left(-\theta\left( p+\frac{C(q_{i-1}+1)}{\mu}\right)\right) \right\}\\
    & \quad +
\sum_{i\in K_1}\log\left\{ \lambda\exp\left(-\theta\left( p+\frac{C(q_{i-1}+1)}{\mu}\right)\right) \right\}+
|K_2|\log\left\{ \mu \right\}.
\end{align*}

Yielding the estimating equation,
\[
\sum_{i\in K_1\cup K_2}\left(\frac{\lambda \left( p+\frac{C(q_{i-1}+1)}{\mu}\right)\exp\left(-\theta\left( p+\frac{C(q_{i-1}+1)}{\mu}\right)\right)}{\mu+\lambda\exp\left(-\theta\left( p+\frac{C(q_{i-1}+1)}{\mu}\right)\right)}\right)=
\sum_{i\in K_1}\left(p+\frac{C(q_{i-1}+1)}{\mu}\right).
\]

We now verify that Assumptions A1-A4 hold. Thus, by Theorem~\ref{theorem1} the estimator is consistent and by Theorem~\ref{theorem2}, the normalized estimation errors are asymptotically normal.

\begin{itemize}

\item A1. Suppose that $\theta_0 \in \Theta$ where $\Theta>0$ is a closed interval and, hence compact and convex.

\item A2.  Follows from the fact that \eqref{eq:exp_F} and\eqref{eq:exp_F_der_theta} are continuous in $\theta\in\Theta$.

\item A3. By \eqref{eq:exp_F_der_theta},
$$\left| \frac{-F_{\theta}'(r(q_i),\theta)}{1-F(r(q_i),\theta)} \right|=\left| \frac{-r(q_i)\exp(-\theta r(q_i))}{\exp(-\theta r(q_i))}  \right|=r(q_i),$$
which is a linear function w.r.t. $q$, and therefore it is integrable as $\mathbb{E}[Q_0]<\infty$.

\item A4. Since $n=1$, the inverse of the $\Sigma$  is simply the reciporal and we just need to verify the integrability condition.  For the exponential case we have that
\[
\left|\frac{\mu\lambda(F_{\theta}'(r(q_i),\theta))^2}{(1-F(r(q_i),\theta))(\mu+\lambda(1-F(r(q_i),\theta)))^2}\right|=
\frac{\mu\lambda r^2(q_i)\exp(-\theta r(q_i))}{(\mu+\exp(-\theta r(q_i)))^2}.
\]
Recall that $r(q)$ is a linear function, and as it is also continuous, it is bounded on any closed interval $[0,q]$. Therefore,
$$\lim_{q\rightarrow \infty}{\frac{\mu\lambda r^2(q_i)\exp(-\theta r(q_i))}{(\mu+\exp(-\theta r(q_i)))^2}}=
\frac{\lim_{q\rightarrow \infty}{\mu\lambda r^2(q_i)\exp(-\theta r(q_i))}}{\lim_{q\rightarrow \infty}{(\mu+\exp(-\theta r(q_i)))^2}}=\frac{\lambda}{\mu}\lim_{q\rightarrow \infty}{r^2(q_i)\exp(-\theta r(q_i))}=0$$
where we used the fact that  $exp(\theta r(q))>>r^2(q)$ for large $q$. By Lemma~\ref{lem:ergodic} we conclude that
\[
\mathbb{E}\left[\frac{\mu\lambda r^2(Q_0)\exp(-\theta r(Q_0))}{(\mu+\exp(-\theta r(Q_0)))^2}\right]<\infty ,
\]

\end{itemize}

\textbf{Simulation analysis.}\label{sec:simulation-analysis.}

Figure 2 illustrates the convergence of
\(\Psi_k(\hat{\theta}_k)\) to the true value \(\Psi(\theta_0)\)=0, a result that was theoretically verified in Lemmas \ref{lem:ergodic} and \ref{lem:unique}. The variance of \(\Psi(\theta)\) is quite large when \(k\) is small, however, it decreases fast as \(k\) grows. 

\begin{figure}[H]

{\centering \includegraphics[width=0.7\linewidth]{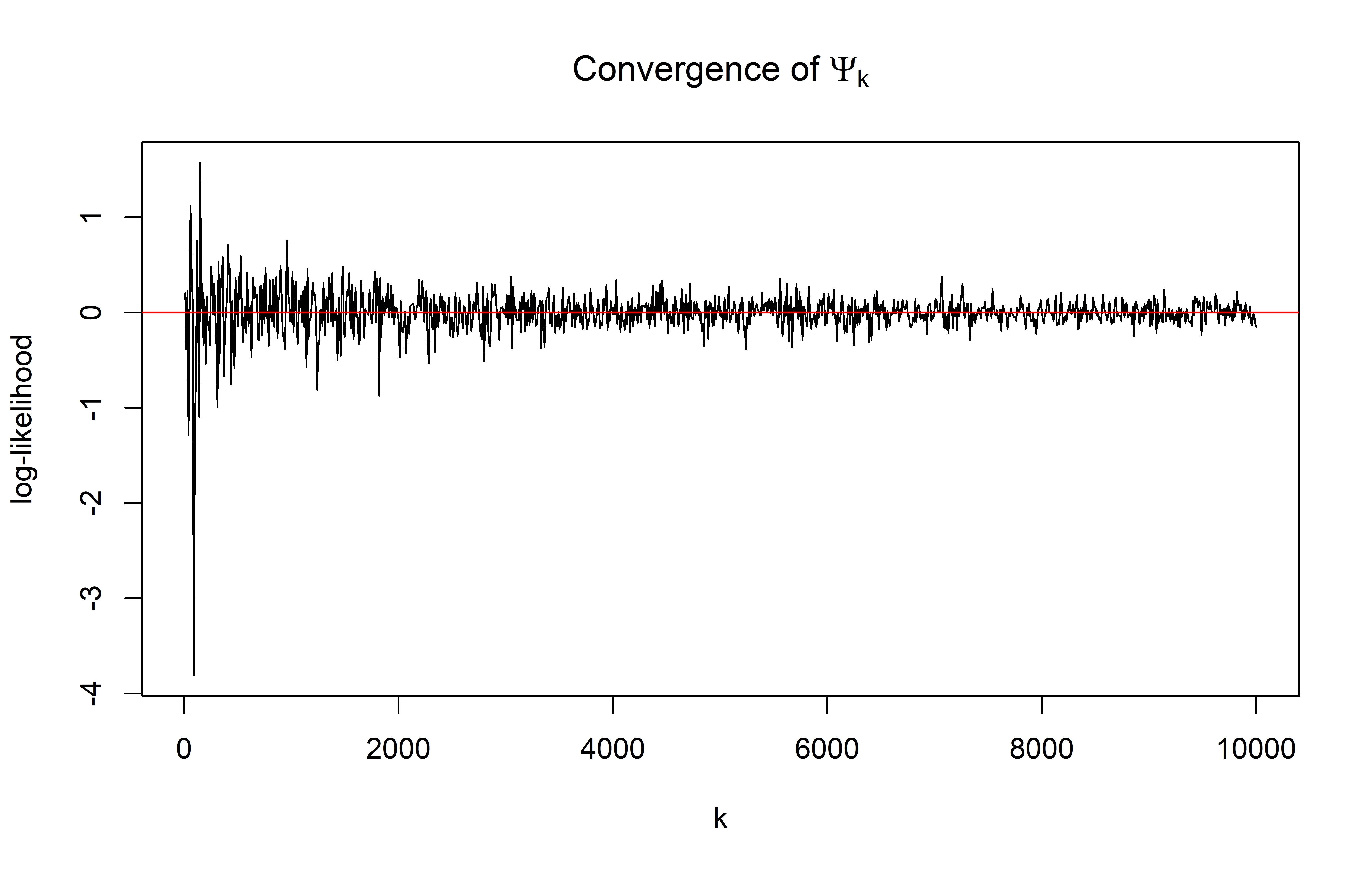} 
}

\caption{Convergence of the derivative of the log-likelihood function ($\Psi_k$). Simulation parameters: $\lambda=1,\: \mu=1,\: \theta=0.02,\: C=1,\: p=15$.}\label{fig:unnamed-chunk-5}
\end{figure}

\begin{figure}[H]

{\centering \includegraphics[width=0.7\linewidth]{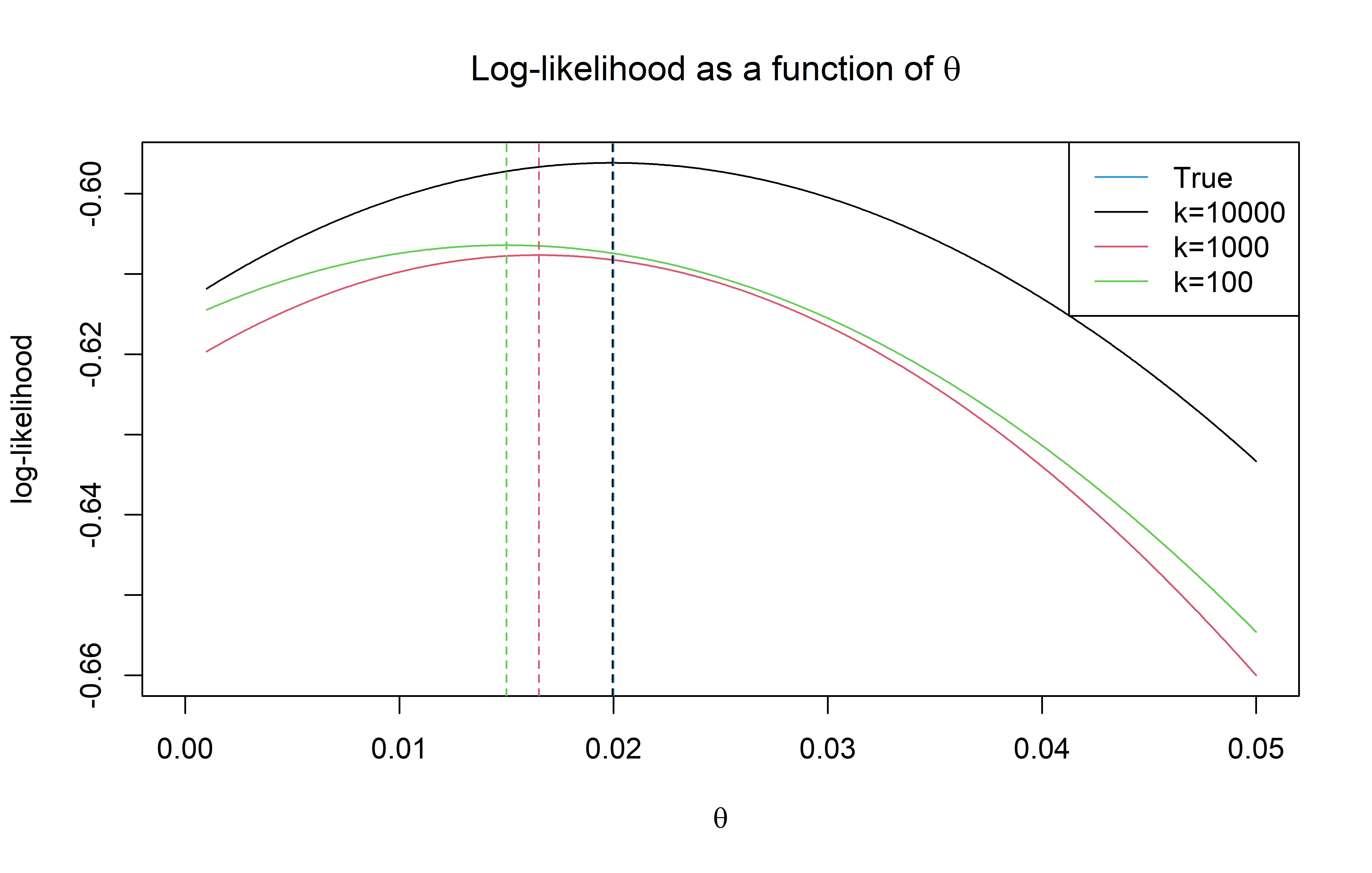} 

}

\caption{Log-likelihood as a function of $\theta$. Simulation parameters: $\lambda=1,\: \mu=1,\: C=1,\: p=15$,\: k=100/1000/10000.}\label{fig:unnamed-chunk-6}
\end{figure}

Figures 3 and 4 demonstrate the result of Theorems \ref{theorem1} and \ref{theorem2}. Namely, as \(k\) grows, the
estimated value of \(\theta\) approaches the true value with a decreasing variance. In Figure 3 the convergence of the log-likelihood and the MLE are displayed. Figure 4 further displays the variance of the estimation error in terms of a confidence interval around the point estimator calculated using the asymptotic variance. Notably, even for small \(k\), the estimated \(\theta\) is within a narrow confidence interval, although the log-likelihood function is still far from the stationary counterpart.

\begin{figure}[h]

{\centering \includegraphics[width=0.7\linewidth]{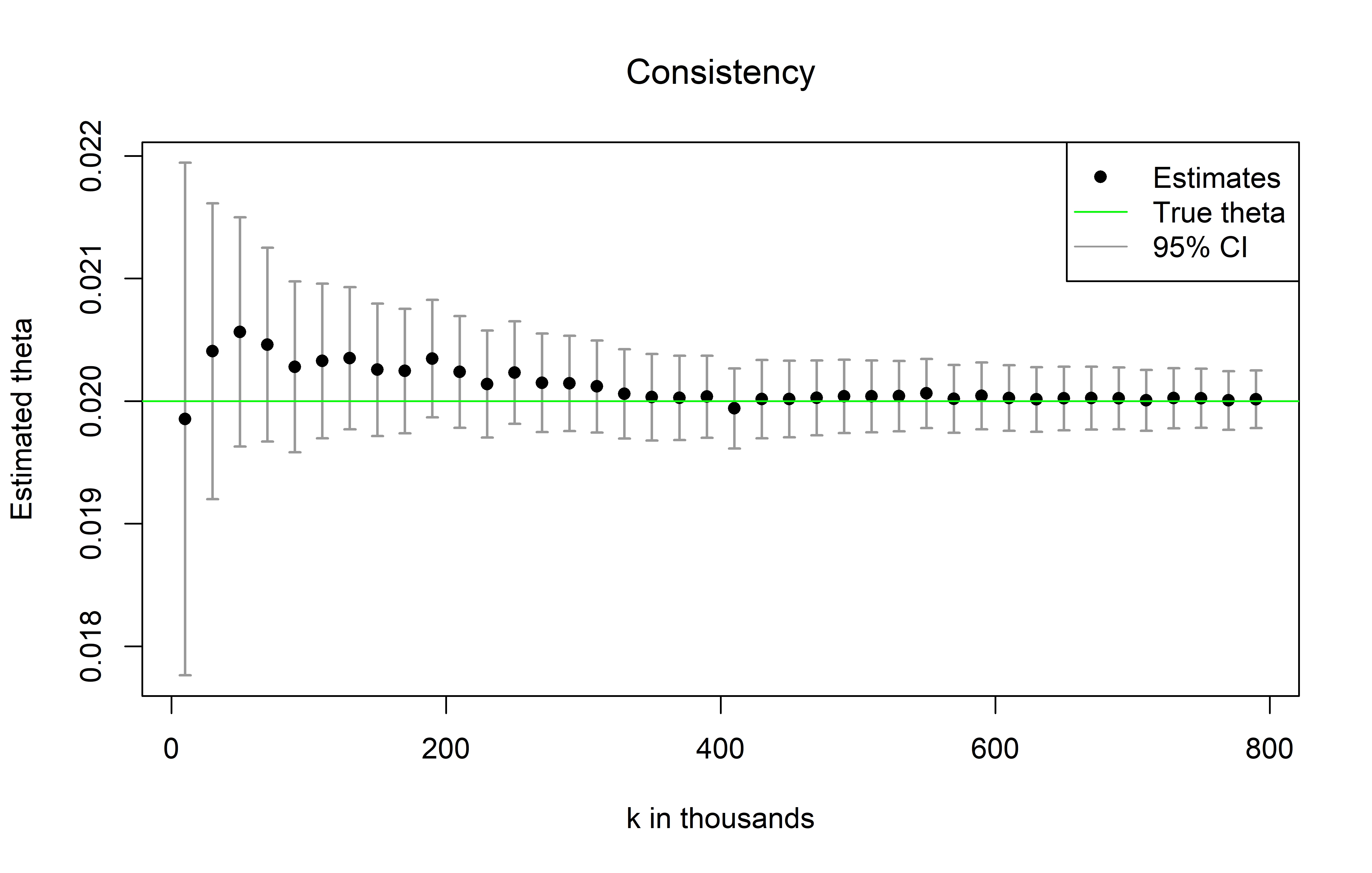}
}

\caption{Convergence of the estimated $\theta$ to the true value as $k$ grows. The confidence Interval is calculated using asymptotic variance in Theorem~\ref{theorem2}. Simulation parameters: $\lambda=1,\: \mu=1,\: \theta=0.02,\: C=1,\: p=15$.}\label{fig:unnamed-chunk-7}
\end{figure}

\begin{figure}[H]

{\centering \includegraphics[width=0.7\linewidth]{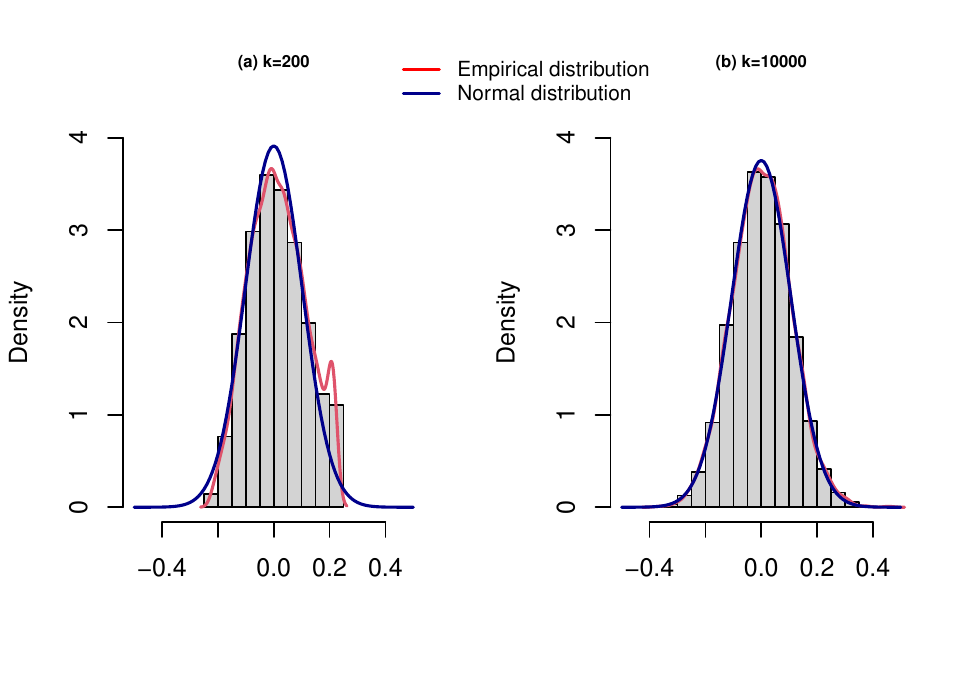} 

}

\caption{Histogram of normalized estimation errors. The red line stands for estimated density. The red line stands for normal distribution with a mean of 0 and a standard deviation the same as that of the estimated distribution. For k=200, the mean value is  0.015  and the p-value of the Shapiro-Wilk test is $\sim$ 0 . For k=10000, the mean value is  0.002  and the p-value of the Shapiro-Wilk test is  0.11. Simulation parameters: $\lambda=1,\: \mu=1,\: \theta=0.02,\: C=1,\: p=15$, repeats=5000.}\label{fig:unnamed-chunk-8}
\end{figure}

Figure 5 demonstrates the main result of Theorem \ref{theorem2}. When the sample is
small, the normalized estimation error distribution is not normal,
neither according to Shapiro-Wilk test nor visually. However, for
\(k=10000\), they are normally distributed around \(0\).

\begin{figure}[H]
\centering \includegraphics[width=0.7\linewidth]{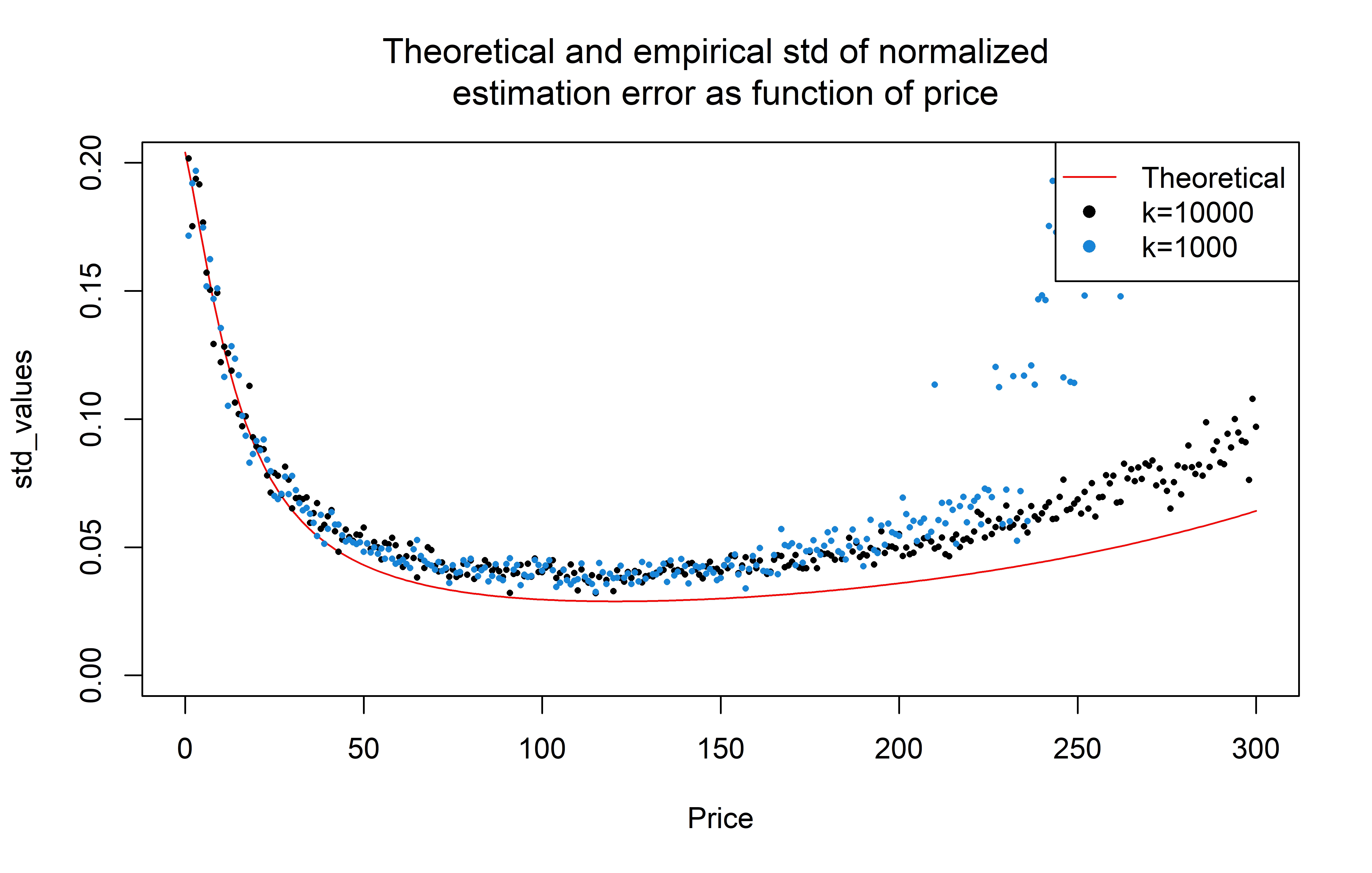} 
\caption{Theoretical and empirical std of normalized estimation error as a function of price. Simulation parameters: $\lambda=1,\: \mu=1,\: C=1,\: \theta=0.02$, \: k=1000/10000, repeats for each price.}\label{fig:unnamed-chunk-9}
\end{figure}

As was shown in Theorem \ref{theorem2}, the variance of asymptotic errors
distribution is
\begin{align*}
    \left(\mathbb{E}\left[\frac{\mu \lambda (F'_{\theta}(Q_0,\theta))^2 }{(1-F(Q_0,\theta))(\mu+\lambda(1-F(Q_0,\theta)))^2}\right]\right)^{-1}.
\end{align*}
The stationary distribution of \(Q_0\) can be calculated with
(\ref{prob_est}) and (\ref{xi}). The actual variance can be calculated
by simulation of \(k\) - steps queue, \(n\) - times. Finally, that can
be done for different prices. Figure 6 compares the theoretical std and
std from such simulation. The probable reason for the difference between
the empirical and theoretical std is that the standard deviation of the
empirical distribution is an upward-biased, but consistent estimator for
standard deviation. As the price grows, the effective sample size is
smaller for a fixed number of steps and estimation is more biased. As a
result, a higher number of steps lead to better estimation even for high
prices.

\begin{figure}[H]
{\centering \includegraphics[width=0.7\linewidth]{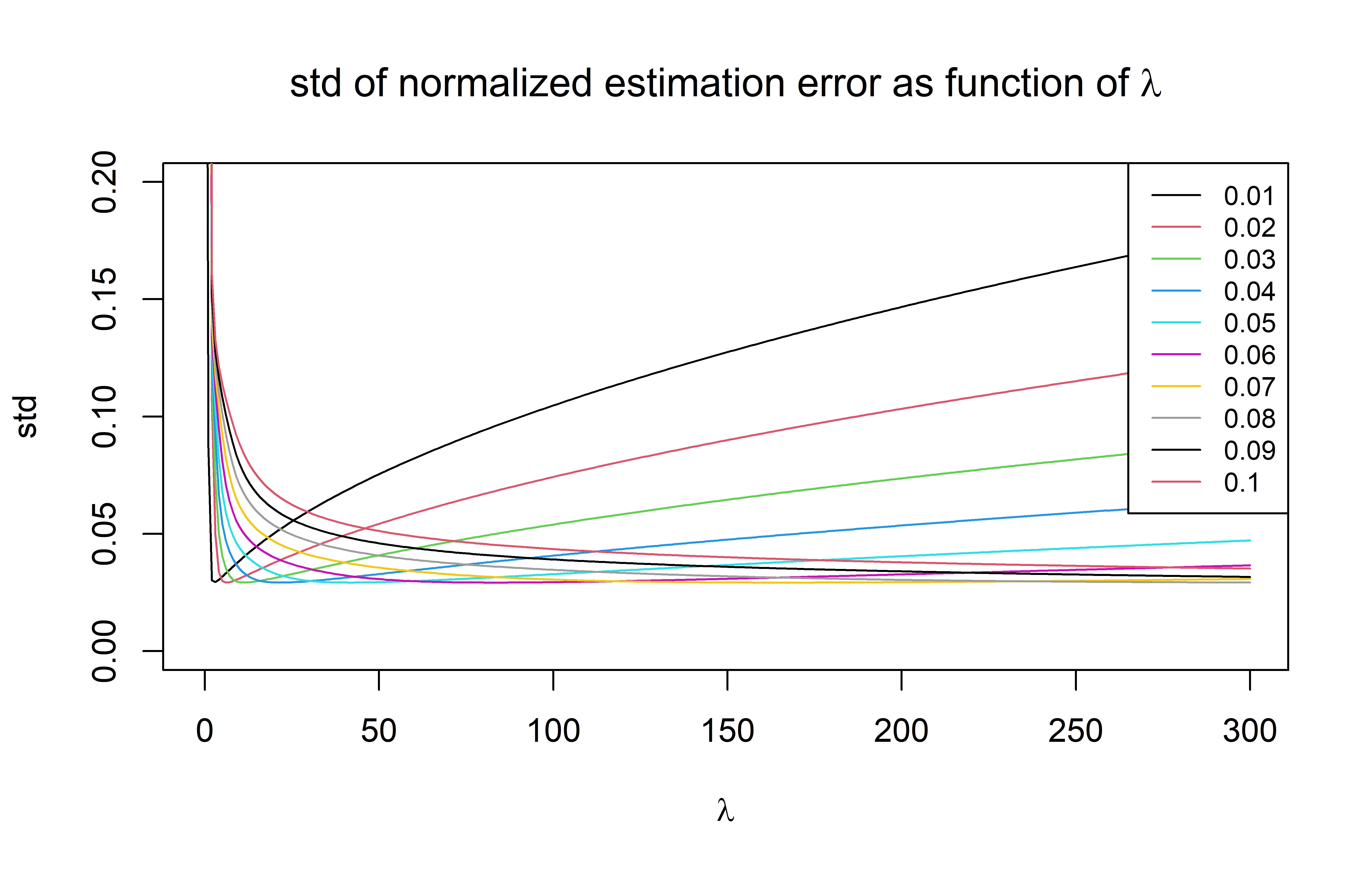} 
}
\caption{ Theoretical std of normalized estimation error as function of $\lambda$ calculated for different values of $\theta$. Simulation parameters: $p=15,\: \mu=1,\: C=1$.}\label{fig:unnamed-chunk-10}
\end{figure}

Figure 7 shows how std changes with $\lambda$, where $\mu=1$ is fixed so $\lambda$ equals to the utilization rate. Similarly to behavior in price (Figure~\ref{fig:unnamed-chunk-9}), std is high when $\lambda$ is very low or very high, with some minimum between. An intuitive explanation for this is that more information is obtained when the queue is observed at many different lengths, thus catching different parts of the service distribution. When $\lambda$ is low, the queue length is small most of the time and the transition data becomes less informative. Similarly, when $\lambda$ is very high, the queue length high most of the time, and again the data become less informative. Moreover, a lower $\theta$, i.e., a higher average reward, strengthens this effect encourages more customers to join at any queue length.

\subsection{Hyperexponential service value.}\label{sec:hyperexponential-service-value.}

We include the hyperexponential distribution as a representative example to illustrate our methodology in a multiparametric setting. Although it describes a single random variable, it is governed by a vector of parameters (mixture weights and rates), making it a natural test case for inference with multiple unknown parameters. Importantly, as a finite mixture of exponentials, the hyperexponential family is dense in the class of distributions on the positive real line, and can therefore approximate a broad range of service-value distributions with arbitrary accuracy (see \cite{botta_approximation_1986}). Consequently, adopting a hyperexponential specification significantly reduces the need for precise prior knowledge of the true service-value distribution, while still allowing for flexible and accurate modeling of heterogeneity. This makes it particularly appealing for practical applications where the underlying distribution is unknown or only partially understood.

Suppose that the service values are hyperexponentially distributed with parameters $\pmb{\theta}=\{\gamma_1,...\gamma_m,w_1,...,w_m\}$. Its CDF is
\begin{align}\label{eq:hexp_F}
F(r,\pmb{\theta})=1-\sum_{j=1}^{m}w_j\exp(-\gamma_j r)\ ,  r\geq 0 .
\end{align}
Where $\gamma_j>0,\: w_j\geq \underline{w}$ and $\sum_{j=1}^{m}w_j=1$, and $\underline{w}>0$ is a uniform lower bounds for the weights. Clearly, the last constraint implies that there are in fact $n=2m-1$ unknown parameters. For the model to be identifiable we also assume that the rate parameters are ordered:
\begin{align*}
    0\leq \gamma_1\leq \gamma_2\leq \ldots\leq \gamma_m .
\end{align*}
Without this assumption you can trivially get the same distribution by switching the labels of two pairs of $(\gamma_j,w_j)$.

We next verify that all of the assumptions required for our asymptotic results hold in this case. This will be followed by analysis of simulation experiments involving estimation.

\textbf{Verification of
assumptions}\label{verification-of-assumptions}

Taking derivatives of \eqref{eq:hexp_F} yields
\begin{align}\label{eq:hexp_F_der_lambda}
\frac{\partial}{\partial\gamma_j}F(r(q_i),\pmb{\theta})=w_jr(q_i)\exp(-\gamma_j r(q_i)),
\end{align}
and
\begin{align}\label{eq:hexp_F_der_w}
\frac{\partial}{\partial w_j}F(r(q_i),\pmb{\theta})=-\exp(-\gamma_j r(q_i)).
\end{align}
The log-likelihood is then
\begin{align*}
    \log \mathcal{L}(\pmb{\theta}|Q_k)&= -\sum_{i\in K_1\cup K_2}\log\left\{\mu+\lambda\sum_{j=1}^{m}w_j\exp(-\gamma_j r(q_i)) \right\}\\
    & \quad +
\sum_{i\in K_1}\log\left\{ \lambda\sum_{j=1}^{m}w_j\exp(-\gamma_j r(q_i)) \right\}+
|K_2|\log\left\{ \mu \right\}.
\end{align*}

Yielding the estimating equation for $\gamma_j$,
\[
\sum_{i\in K_1\cup K_2}\left(\frac{\lambda r(q_i)\exp(-\gamma_jr(q_i))}{\mu+\lambda\sum_{l=1}^{m}w_l\exp(-\gamma_l r(q_i))}\right)=
\sum_{i\in K_1}\left(\frac{r(q_i)\exp(-\gamma_jr(q_i))}{\sum_{l=1}^{m}w_l\exp(-\gamma_l r(q_i))}\right),\:j=1,...,m.
\]

and for $w_j$,

\[
\sum_{i\in K_1\cup K_2}\left(\frac{\lambda\exp(-\gamma_jr(q_i))}{\mu+\lambda\sum_{l=1}^{m}w_l\exp(-\gamma_l r(q_i))}\right)=
\sum_{i\in K_1}\left(\frac{\exp(-\gamma_jr(q_i))}{\sum_{l=1}^{m}w_l\exp(-\gamma_l r(q_i))}\right),\:j=1,...,m-1 .
\]

Note that when solving the estimation equations we have an additional constraint,  namely $\sum_{j=1}^{m}w_j=1$, to ensure the solution is within the correct parameter space.

Now, we verify that Assumptions A1-A4 hold.

\begin{itemize}

\item A1. Suppose that $\pmb{\theta}_0 \in \pmb{\Theta}$, where is $\pmb{\Theta}$ compact and convex. For example, the rates can belong to some interval $[0,\bar{\gamma}]$ and the weights form a simplex. Recall that we also assumed $w_j\geq \underline{w}>0$ for all $j=1,\ldots,m$.

\item A2. Observe that \eqref{eq:hexp_F}, \eqref{eq:hexp_F_der_lambda}, \eqref{eq:hexp_F_der_w} are all continuous w.r.t. $\pmb{\theta}\geq\pmb{0}$.

\item A3. For any $j\in \{1,...,m\}$, as $w_j\in[0,1]$,

$$\left|\frac{\frac{\partial}{\partial\gamma_j}F(r(q),\pmb{\theta})}{1-F(r(q),\pmb{\theta})}\right|=
\left|-\frac{w_jr(q)\exp(-\gamma_jr(q))}{\sum_{l=1}^{m}w_l\exp(-\gamma_l r(q))}\right|\leq r(q),$$

and 

$$\left|\frac{\frac{\partial}{\partial w_j}F(r(q),\pmb{\theta})}{1-F(r(q),\pmb{\theta})}\right|=
\left|-\frac{\exp(-\gamma_jr(q))}{\sum_{l=1}^{m}w_l\exp(-\gamma_l r(q))}\right|
\leq \frac{\exp(-\gamma_jr(q))}{w_j\exp(-\gamma_j r(q))}
\leq \frac{1}{w_j}\leq \frac{1}{\underline{w}}<\infty.$$
 Thus,
\begin{align*}
\max_{j=1,\ldots,n}\left|\frac{ \nabla F(r(q),\pmb{\theta})}{1-F(r(q),\pmb{\theta})}\right|_j\leq \max\left\lbrace\frac{1}{\underline{w}},r(q)\right\rbrace\ , \forall q\geq 0 ,
\end{align*}
and $\mathbb{E}[Q_0]<\infty$ implies that $\mathbb{E}[ \max\left\lbrace\frac{1}{\underline{w}},r(q)\right\rbrace]<\infty$ as Assumption A3 requires.

\item A4. Firstly we calculate the $\Sigma$. Recall that the parameter vector is now $\{\gamma_1,\ldots,\gamma_m,w_1,\ldots,w_{m-1}\}$. For the sake of brevity, we rewrite the covariance matrix in a manner that includes that actual parameters and not their index (out of the $2m$ parameters); 
\begin{align*}
    \Sigma_{\gamma_j,\gamma_k}& :=\Sigma_{j,k} \ , k,j=1,\ldots,m,\\
    \Sigma_{\gamma_j,w_k}& :=\Sigma_{j,m+k} \ ,  k=1,\ldots,m\ , j=1,\ldots,m-1,\\
    \Sigma_{w_j,w_k}& :=\Sigma_{m+j,m+k} \ ,  k=1,\ldots,m\ , j=1,\ldots,m-1.
\end{align*}
Applying (\ref{eq:hexp_F}), (\ref{eq:hexp_F_der_lambda}) and (\ref{eq:hexp_F_der_w}) into into the derivation in Assumption~A4 and have that
$$\Sigma_{\gamma_j,\gamma_k}=\mathbb{E}\left[
\frac{\mu\lambda w_j w_k r^2(Q_0) \exp(-(\gamma_j+\gamma_k)r(Q_0))}{\sum_{l=1}^{m}w_l\exp(-\gamma_l r(Q_0))(\mu+\lambda \sum_{l=1}^{m}w_l\exp(-\gamma_l r(Q_0)))^2}
\right],$$

$$\Sigma_{\gamma_j,w_k}=\mathbb{E}\left[
\frac{-\mu\lambda w_j r(Q_0) \exp(-(\gamma_j+\gamma_k)r(Q_0))}{\sum_{l=1}^{m}w_l\exp(-\gamma_l r(Q_0))(\mu+\lambda \sum_{l=1}^{m}w_l\exp(-\gamma_l r(Q_0)))^2}
\right],$$

$$\Sigma_{w_j,w_k}=\mathbb{E}\left[
\frac{\mu\lambda \exp(-(\gamma_j+\gamma_k)r(Q_0))}{\sum_{l=1}^{m}w_l\exp(-\gamma_l r(Q_0))(\mu+\lambda \sum_{l=1}^{m}w_l\exp(-\gamma_l r(Q_0)))^2}
\right].$$

Now we use the fact that $\lim_{q\rightarrow \infty}(\mu+\lambda \sum_{l=1}^{m}w_l\exp(-\gamma_l r(Q_0)))^2=\mu^2$ and
$$
\exp((\gamma_j+\gamma_k)r(q))>>
\sum_{l=1}^{m}w_l\exp(\gamma_l r(q))>>
r(q)^2>>
r(q)
$$
for large $q$. Thus, the terms inside all three expectations are bounded by a constant with respect to $q$, and are therefore integrable with respect $Q_0$. Then by Lemma~\ref{lem:ergodic} we conclude that $\Sigma_{\gamma_j,\gamma_k}$, $\Sigma_{\gamma_j,w_k}$ and $\Sigma_{w_j,w_k}$ are finite. 

The random variables inside the expectations of the elements of $\Sigma$ are all integrable with respect to the s.tationary distribution. Furthermore, as they are not linearly dependent (recalling that we have $m-1$ weight parameters), we conclude tha $\Sigma$ is invertible and A4 holds.

\end{itemize}

\textbf{Simulation analysis}

We focus on the case \(m=2\), which is the simplest nontrivial vector-valued setting. 
Figure~\ref{fig:ll-hyperexponential} depicts the log-likelihood function as a function of 
\(\pmb{\gamma}\), with the weight vector \(\pmb{w}\) fixed at its true value. 
Several observations emerge from this figure. 
First, even for a sample size of \(k=10^5\), the maximizer is attained at 
\(\pmb{\gamma}=\{0.073,\,0.061\}\) and \(\pmb{w}=\{0.091,\,0.908\}\), whereas the 
true parameters are \(\pmb{\gamma}=\{0.1,\,0.01\}\) and 
\(\pmb{w}=\{0.3,\,0.7\}\). 
This illustrates that accurate estimation may require \(10^7\) or even \(10^8\) observations, 
depending strongly on the specific parameter configuration. 
Second, the log-likelihood is extremely flat in the direction of the smaller rate parameter. 
In particular, black points corresponding to values within \(10^{-5}\) of the maximum (for the 
true \(\pmb{w}\)) span a wide region. 
This pronounced flatness makes direct numerical maximization practically infeasible. 
In contrast, solving the estimating equations yields satisfactory results. 
However, since these equations are nonlinear, they must be solved iteratively, and the flatness 
of the objective function can lead to slow convergence and substantial computational burden. 
Finally, when the separation between the rate parameters increases, the log-likelihood becomes 
non-unimodal. 
Nevertheless, the symmetry induced by the line \(\gamma_1=\gamma_2\) partitions the parameter 
space into two regions, so that initializing the algorithm once with \(\gamma_1<\gamma_2\) and 
once with \(\gamma_1>\gamma_2\) guarantees convergence to the global maximum.

\begin{figure}[H]
{\centering \includegraphics[width=0.8\linewidth]{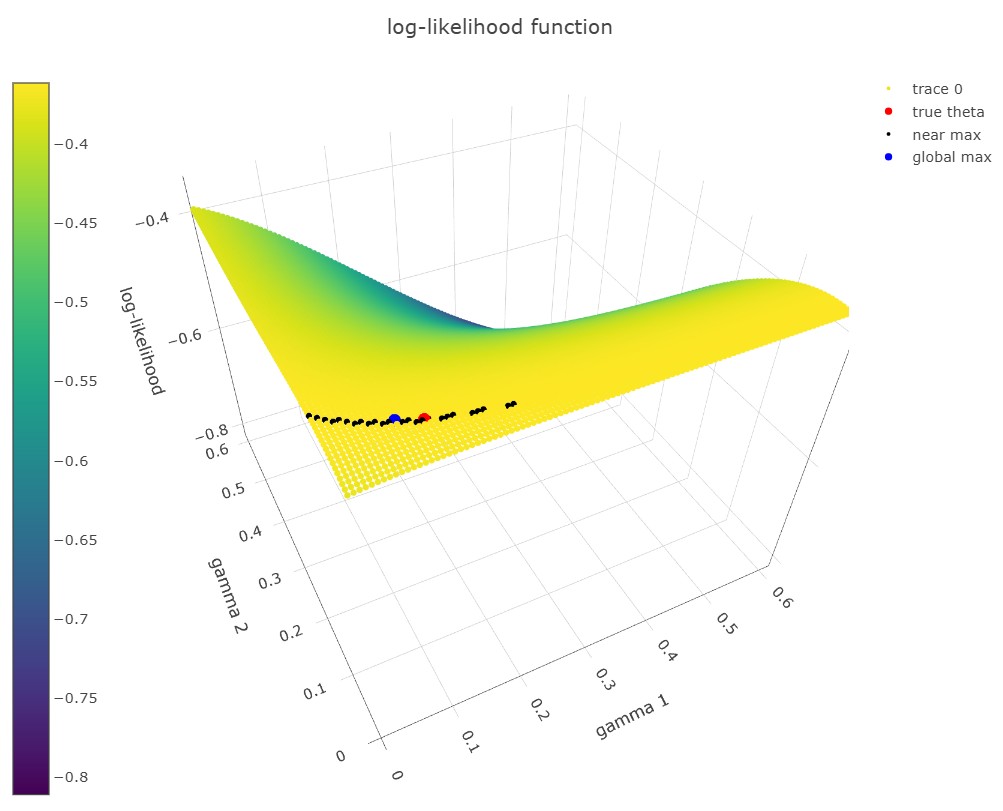}
}
\caption{Log-likelihood function in $\gamma_1,\: \gamma_2$ with fixed true $w_1,\:w_2$. Simulation parameters: $\lambda=0.5,\: \mu=1,\: \pmb{\gamma}=\{0.1,\:0.05\},\: \pmb{w}=\{0.3,\: 0.7\},\: C=1,\: p=5,\: k=10^6$. The global maximum is attained at $\pmb{\gamma}=\{0.073,0.061\},\:\pmb{w}=\{0.091,\: 0.908\}$}.\label{fig:ll-hyperexponential}
\end{figure}

As Figure~\ref{fig:consistency-hyperexponential} illustrates, three of the four parameters require on the order of $4\cdot 10^5$
 observations to achieve reliable estimation. Overall, this confirms that accurate inference in the hyperexponential case demands a very large sample size. Moreover, even when sufficient data are available, obtaining the estimator involves a lengthy iterative procedure. Consequently, the maximizer of the likelihood function in the parameter space may lie substantially closer to the true parameter vector than the numerical estimate produced by the algorithm.

One explanation for this behavior is that very similar distribution functions can be generated by different combinations of the parameter vector \(\pmb{\theta}\). While this makes the hyperexponential distribution a flexible parametric model, this also means that accurate estimation of the individual parameters is challenging. 
On the other hand, for practical applications the exact parameter values are often less important than the distribution function they induce. For this reason, as we demonstrate below, the pricing algorithm performs well with a substantially smaller sample size than that required for accurate estimation of \(\pmb{\theta}\).

It can also be observed that the confidence intervals for three of the four parameters are extremely narrow, while the true parameter values frequently lie outside these intervals. We recall that the confidence intervals are computed using the asymptotic variance formula, which relies on the assumption that the queue-length process has reached stationarity. However, as demonstrated earlier, convergence to stationarity under a hyperexponential service-value distribution can be slow. As a result, for moderate values of k, the finite-sample variance of the estimator may be significantly larger than its asymptotic counterpart, leading to undercoverage of the nominal confidence intervals.

\begin{figure}[H]
{\centering \includegraphics[width=1.0\linewidth]{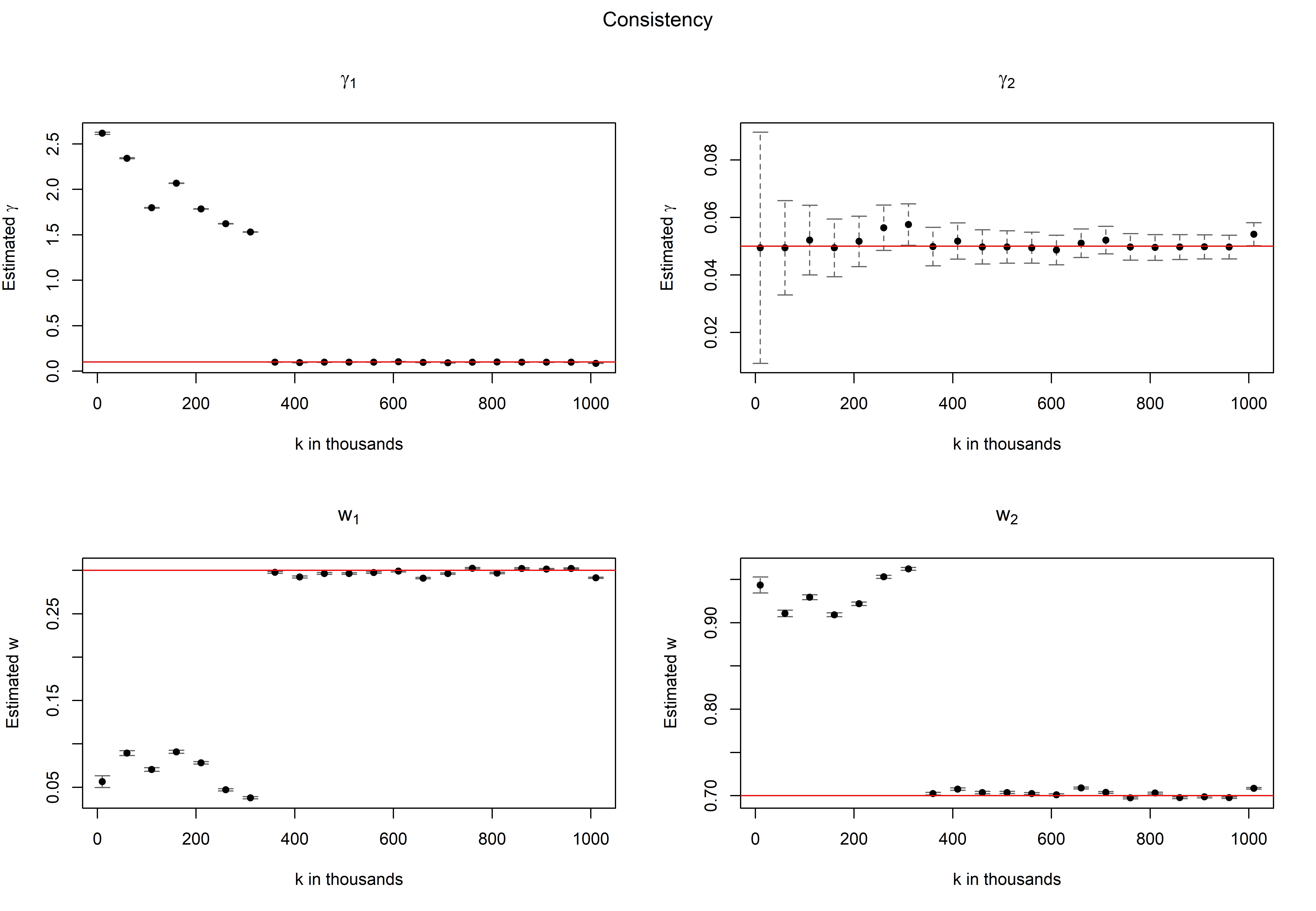}
}
\caption{The plots show the convergence of the estimated $\pmb{\theta}$ to the true value as $k$ grows. Confidence Intervals are calculated using asymptotic variance formula. Simulation parameters: $\lambda=0.5,\: \mu=1,\: \pmb{\gamma}=\{0.1,\:0.05\},\: \pmb{w}=\{0.3,\: 0.7\},\: C=1,\: p=5$.}\label{fig:consistency-hyperexponential}
\end{figure}

\section{Iterative price optimization algorithm}\label{sec:iterative-pricing-algorithm}

The estimation of utility parameters can be used for different purposes.
One of them is the optimization of the admission price. Suppose that an
administrator is interested to maximize the expected revenue by choosing
the best service price. For any price \(p>0\), the stationary expected
revenue per unit of time \(\Pi\) is given by

\[\Pi(p,\pmb{\theta})=p\sum_{q=0}^{\infty}P_{\theta}(Q_0=q)\lambda_q\:.\]

The stationary distribution \(Q_0\) can be calculated with (Haviv
(2009), section 8.3.1), \begin{align}\label{prob_est}
P_{\theta}(Q_0=q)=\frac{\xi_q}{\sum_{q=0}^{\infty}\xi_q},
\end{align}

where

\begin{align}\label{xi}
\xi_q=
\left\{\begin{matrix}
1 &,\:\: q=0  \\
\prod_{j=1}^{q}\frac{\lambda_{j-1}}{\mu}  &,\:\: q>0
\end{matrix}\right.
\end{align}

In practice, the infinite sum is evaluated by truncation with high
enough \(q^*\), where
\(q^*=\left\{ q\in\mathbb{Z}^+:p(Q>q)<\epsilon \right\}\), and
\(\epsilon\) is a specified tolerance parameter.

The expected revenue can be computed only assuming some value of
\(\pmb{\theta}\), which is unknown and should be estimated. However, in
order to estimate \(\pmb{\theta}\), a sample must be collected. In order
to collect the sample, the administrator needs to decide first on some
price. One possible option is to choose a random price and collect as
much data as possible. However, if the chosen price is far away from the
optimal one it can lead to substantial losses in revenue. Another
possibility is to optimize the price in steps. We suggest the following
iterative optimization algorithm.

\hfill\break

\begin{algorithm}[H]
\caption{Iterative price optimization algorithm}\label{algo}
\begin{algorithmic}[1]

\Start{\:Algorithm}
  \State Choose an arbitrary price $p_1$.
  \State Choose first sample size $k_1$.
  \State Choose strictly increasing function $g(k):\mathbb{N}\rightarrow \mathbb{N}$.
  \State Set $\Delta_1<tol$
  \State Set $i=1$
  \While{$\Delta_i<$ tol}
    \State Set $p_{i}=p_1$ if $i=1$ and $p_{i}=\hat{p}_{i-1}$ otherwise.
    \State Set $k_{i}=k_1$ if $i=1$ and $k_{i}=g(k_{i-1})$ otherwise.
    \State Collect $k_{i}$ queue length observations.
    \State Estimate $\hat{\pmb{\theta}}_{k_i}$.
    \State Set $\hat{\pmb{\theta}}=\frac{\sum_{j=1}^{i} k_j \hat{\pmb{\theta}}_{k_j}}{\sum_{j=1}^{i} k_j}$.
    \State Find $\hat{p}_{i}=\underset{p}{\text{argmax}}\left\{ \Pi(p,\hat{\pmb{\theta}}) \right\}$.
    \State Set $i=i+1$
    \State Calculate $\Delta_i$.
  \EndWhile
  \State Choose $\hat{p}_{i}$ as optimal price $p^*$.

\end{algorithmic}
\end{algorithm}

\hfill\break

It is important that the number of observations increases every
iteration. The consistency of the estimator ensures that if the
increasing sequence (\(k_i\)) is unbounded then the algorithm converges
in probability to the optimal price (if it is never stopped). The
measure used in the stopping rule (\(\Delta\)) can be defined in
different terms. The choice of \(\hat{\pmb{\theta}}\) as the mean of all
estimated \(\hat{\pmb{\theta}}_{k_i}\) weighted by sample size is
intended to reduce the probability that \(\Delta\) is small or even
\(0\) when the sample size is very small and two consecutive samples can
carry the same information by chance. The \(\Delta\) we used in our
example is the following

\[\Delta_i=\frac{\left| \frac{\pi_i}{t_i} -\Pi(\hat{p},\hat{\pmb{\theta}}) \right|}{\frac{\pi_i}{t_i}}.\]

Where \(\pi_i\) and \(t_i\) are the actual revenue and the total
interarrival time of step \(i\) respectively. A better comparison is
\(\left|\frac{\Pi(\hat{p}, \pmb{\theta})-\Pi(\hat{p},\hat{\pmb{\theta}})}{\Pi(\hat{p}, \pmb{\theta})}\right|\),
which measures a difference between expected stationary revenue with
chosen price (with estimated \(\pmb{\theta}\)) and the actual stationary
revenue with chosen price (with true \(\pmb{\theta}\)). However, in
order to calculate \(\Pi(\hat{p}, \pmb{\theta})\) the true value of
\(\pmb{\theta}\) is required. Thus we replace it with empirical version
\(\frac{\pi_i}{t_i}\). Unlike the \(\Pi(\hat{p}, \pmb{\theta})\),
\(\frac{\pi_i}{t_i}\) is stochastic and depends on the data. As a
result, \(\Delta\) can be very small even though its stationary version
is large enough. However, as \(k\) grows, we expect that
\(\frac{\pi_i}{t_i}\rightarrow\Pi(\hat{p}, \pmb{\theta})\) and the
estimated optimal price will be close to the true one.

\subsection{Exponential example}

We now apply the suggested pricing algorithm to the exponential
example. Figure \ref{fig:unnamed-chunk-11}, shows that when the price is too low or too high the expected revenue decreases,
however somewhere exists an optimal price. And again, the revenue
declines much faster for prices lower than optimal, rather than for
prices higher than optimal. This analysis leads to the conclusion that
if there are no prior beliefs about the optimal price, the learning
should be started from high prices. This strategy will lead to both
lower revenue losses per customer and faster convergence. However, as we
show further, this analysis does not include important information which
affects the choice of the initial price. 

\begin{figure}[h]
{\centering \includegraphics[width=0.7\linewidth]{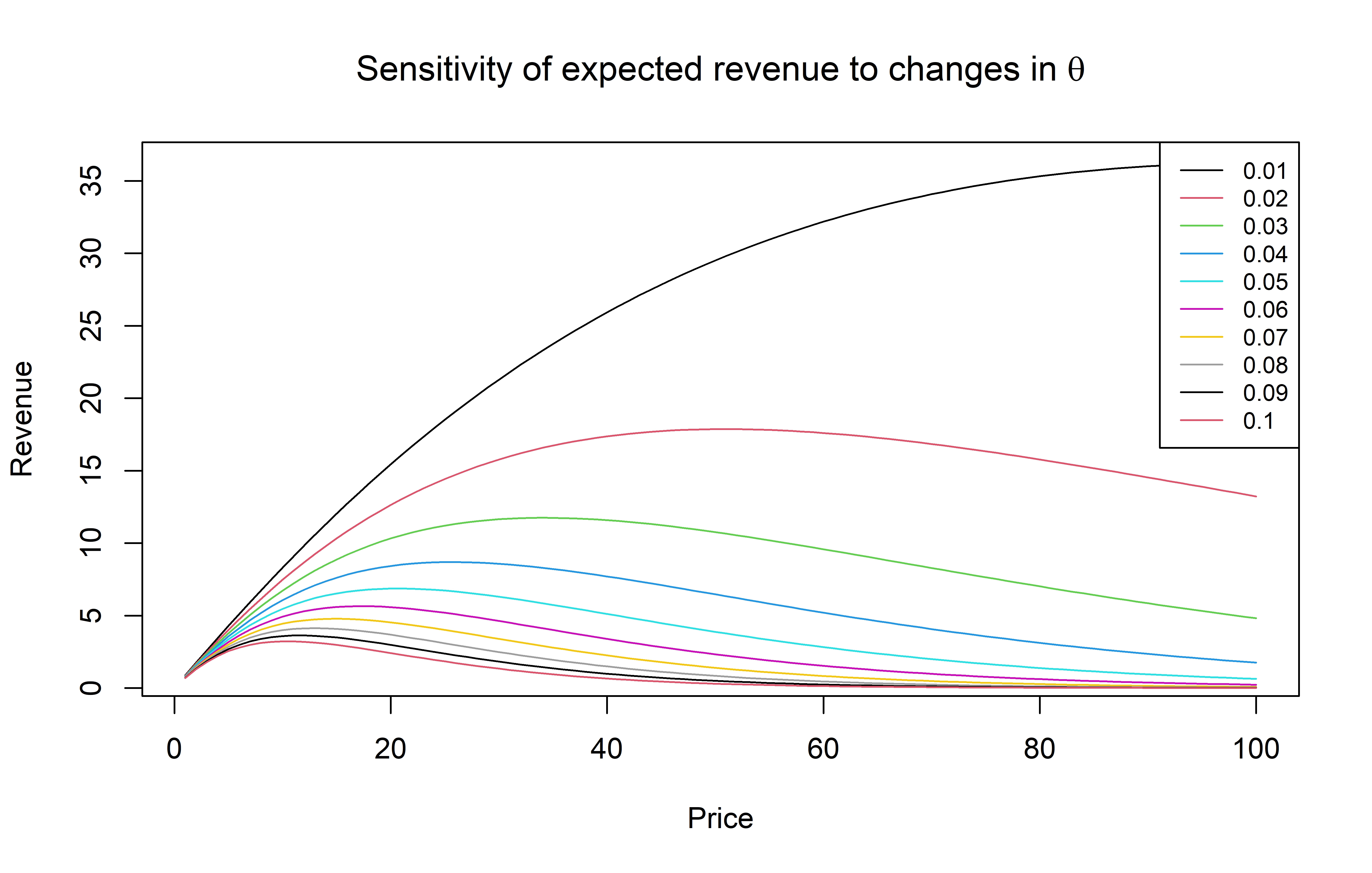} 
}
\caption{Sensitivity of expected revenue to changes in $\theta$. Dashed lines show the prices that maximize revenue for $\theta=0.02$ (50.9) and $\theta=0.08$ (13). Simulation parameters: $\lambda=1,\: \mu=1,\: C=1$. }\label{fig:unnamed-chunk-11}
\end{figure}

Table \ref{tab1} summarizes 100 simulations with different initial
conditions: initial observations number (\(k_1\)), initial
price (\(p_1\)) with \(k_{i+1}=g(k_i)=2k_i\) (the functions used in order to calculate
the \(k_i\) for every iteration). 


In the summary table, the final
stationary fraction of maximum revenue is defined as
\(\frac{\Pi(\hat{p}_l,\theta)}{\Pi(p,\theta)}\) where \(\hat{p}_l\) is
the last price estimated before we stopped the learning and \(p\) is the
optimal price. Stationary cumulative fraction of maximum revenue is defined as
\(\sum_{i=1}^{l}t_i\Pi(\hat{p}_i,\theta)/\sum_{i=1}^{l}t_i\Pi(p,\theta)\).
Where \(\sum_{i=1}^{l}t_i\Pi(p,\theta)\) is the maximum possible revenue
of all iterations before the learning stopped, while
\(\sum_{i=1}^{l}t_i\Pi(\hat{p}_i,\theta)\) is the received revenue
during all iterations before the learning stopped. The total lost
revenue is defined as
\(\sum_{i=1}^{l}\left(t_i\Pi(p,\theta)-t_i\Pi(\hat{p}_i,\hat{\theta_i})\right)\).

\begin{table}[h]

\centering
\resizebox{\textwidth}{!}{%
\begin{tabular}{@{}|l|cccccccc|@{}}
\toprule
$g(k_i)$                                   & \multicolumn{8}{c|}{$2k_i$}                                                                                                                                                                                          \\ \midrule
$k_1^{min}$                                    & \multicolumn{4}{c|}{2}                                                                                            & \multicolumn{4}{c|}{100}                                                                      \\ \midrule
$p_1$                                     & \multicolumn{1}{c|}{1}     & \multicolumn{1}{c|}{15}    & \multicolumn{1}{c|}{100}   & \multicolumn{1}{c|}{250}   & \multicolumn{1}{c|}{1}     & \multicolumn{1}{c|}{15}    & \multicolumn{1}{c|}{100}   & 250    \\ \midrule
Iterations                     & \multicolumn{1}{c|}{7.62}  & \multicolumn{1}{c|}{8.17}  & \multicolumn{1}{c|}{7.96}  & \multicolumn{1}{c|}{6.32}  & \multicolumn{1}{c|}{4}     & \multicolumn{1}{c|}{4}     & \multicolumn{1}{c|}{4}     & 4      \\ \midrule
Total number of observations used for learning          & \multicolumn{1}{c|}{1507}  & \multicolumn{1}{c|}{1547}  & \multicolumn{1}{c|}{1584}  & \multicolumn{1}{c|}{1512}  & \multicolumn{1}{c|}{1502}  & \multicolumn{1}{c|}{1500}  & \multicolumn{1}{c|}{1500}  & 1533   \\ \midrule
Final stationary fraction of max revenue              & \multicolumn{1}{c|}{0.994} & \multicolumn{1}{c|}{0.994} & \multicolumn{1}{c|}{0.996} & \multicolumn{1}{c|}{0.996} & \multicolumn{1}{c|}{0.983} & \multicolumn{1}{c|}{0.991} & \multicolumn{1}{c|}{0.995} & 0.996  \\ \midrule
Stationary cumulative fraction of max   revenue & \multicolumn{1}{c|}{0.978} & \multicolumn{1}{c|}{0.979} & \multicolumn{1}{c|}{0.976} & \multicolumn{1}{c|}{0.636} & \multicolumn{1}{c|}{0.915} & \multicolumn{1}{c|}{0.955} & \multicolumn{1}{c|}{0.95}  & 0.293  \\ \midrule
Total lost revenue                                & \multicolumn{1}{c|}{1164}  & \multicolumn{1}{c|}{992}   & \multicolumn{1}{c|}{1301}  & \multicolumn{1}{c|}{27451} & \multicolumn{1}{c|}{3223}  & \multicolumn{1}{c|}{2083}  & \multicolumn{1}{c|}{2507}  & 128602 \\ \midrule
Mean error of final price                                  & \multicolumn{1}{c|}{2.75}  & \multicolumn{1}{c|}{2.89}  & \multicolumn{1}{c|}{2.64}  & \multicolumn{1}{c|}{2.49}  & \multicolumn{1}{c|}{3.89}  & \multicolumn{1}{c|}{2.94}  & \multicolumn{1}{c|}{2.22}  & 2.51   \\ \midrule
Std of optimal price error                        & \multicolumn{1}{c|}{2.18}  & \multicolumn{1}{c|}{2.18}  & \multicolumn{1}{c|}{1.99}  & \multicolumn{1}{c|}{1.98}  & \multicolumn{1}{c|}{3.11}  & \multicolumn{1}{c|}{2.34}  & \multicolumn{1}{c|}{1.82}  & 1.89   \\ \midrule
Optimal price & \multicolumn{8}{c|}{50.89} \\ \bottomrule
\end{tabular}%
}
\caption{Summary table of 100 simulations for different $p_1$, $k_1^{min}$ with $g(k_i)=2k_i$.}
\label{tab1}
\end{table}

Figures \ref{fig:price_convergence} and \ref{fig:revenue_convergence} demonstrate convergence of price and revenue to the optimal values for different combinations of $p_1$ and $k_1$. Each
point is one iteration, however, the \(x\) axis unit is the total number
of used observations until the iteration (including). This allows us to
compare the learning process of different combinations. The \(y\) axis
of the figures represents the metric per iteration. For example, in
Figure \ref{fig:price_convergence}, the first point of the combination
\(\: p_1=15,\: k_1=100)\) is located on coordinates
\((100, 15)\) which is the initialization point. The second point is
\((300, 52)\) which means that \(k_2=200\) and \(p_2=52\). The \(y\)
axis in Figure \ref{fig:revenue_convergence} is the fraction of maximum revenue, defined as
\(\frac{\Pi(\hat{p}_i,\theta)}{\Pi(p,\theta)}\) where \(\hat{p}_i\) is
the price estimated in iteration \(i-1\).

From the summary table (\ref{tab1}) and convergence plots (Figures \ref{fig:price_convergence} and \ref{fig:revenue_convergence}) we can make a number of conclusions. On the one hand, higher initial price and higher $k_{1}$ (each of them separately) on average lead to faster convergence. On the other hand, faster learning with a high \(p_1\) is achieved at the expense of the revenue gained during the learning period. The higher the price,
the lower the effective arrival rate, the more time required to gather
the desired number of observations, and the higher the revenue loss. As
a result, requiring more observations enhance this effect. The too-low
price also leads to higher loss, though, the desired sample will be
gathered in a short time, thus the total loss will be much smaller (the
loss can be much higher if production cost is taken into account). This
result disagrees with our previous observation of revenue as a function of
the price which didn't take the effect of the time required to gather
the data into account.

\begin{figure}[h]

{\centering \includegraphics[width=0.7\linewidth]{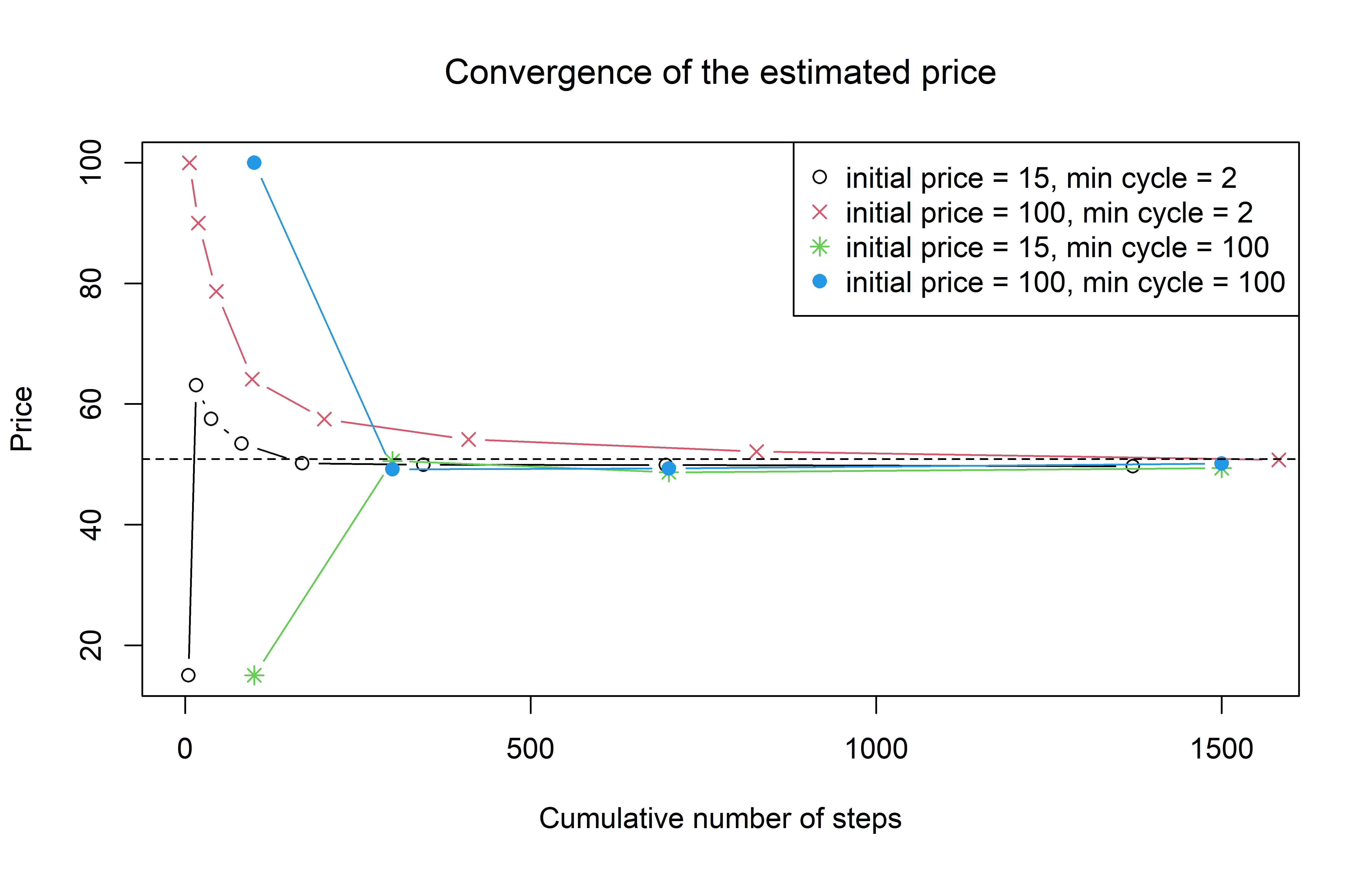} 
}

\caption{Convergence of the estimated price for different initial prices, minimum observation numbers, and observation number growth functions (average of 100 independent runs for each combination of parameters). Optimal price: $50.89$. Simulation parameters: $\lambda=1,\: \mu=1,\: C=1,\: \theta=0.02$.}\label{fig:price_convergence}
\end{figure}
\hfill\break

\begin{figure}[H]

{\centering \includegraphics[width=0.7\linewidth]{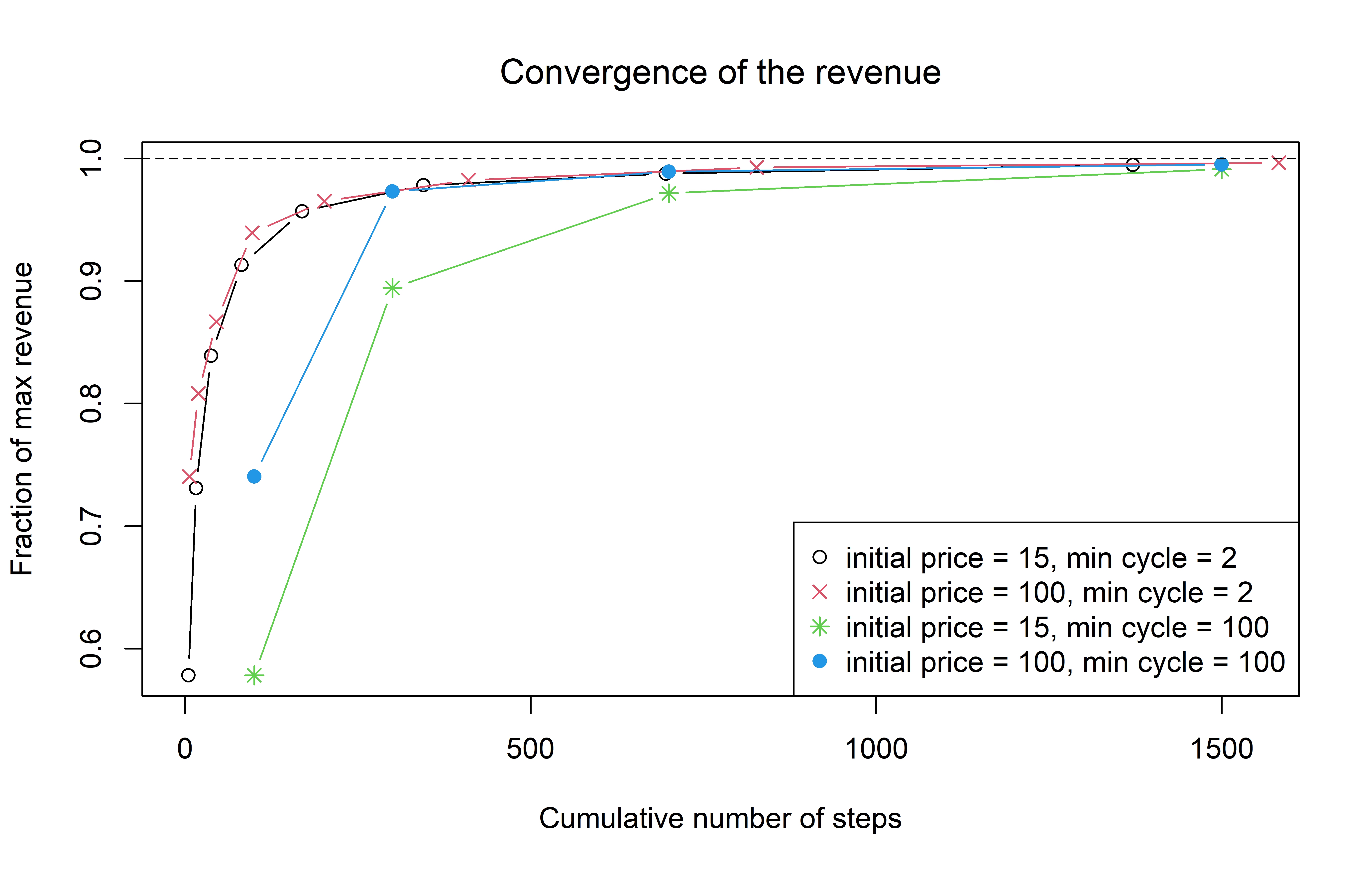}
}

\caption{Convergence of the revenue to the maximum revenue for different initial prices, minimum observation numbers, and observation number growth functions (average of 100 independent runs for each combination of parameters). Simulation parameters: $\lambda=1,\: \mu=1,\: C=1,\: \theta=0.02$.}\label{fig:revenue_convergence}

\end{figure}
As a general guideline, we can conclude that as we expected, for faster
convergence (in sample size terms) it is better to use more observations
and high prices. However, if too high price was chosen by chance,
the lost revenue will be very high. This ref

\subsection{Hyperexponential example}

The pricing algorithm is now applied to the case of hyperexponential service value with $m=2$. Figure~\ref{fig:expected_revenue_hyperexponential} shows that even small perturbations in the parameter vector can induce noticeable changes in the service-value distribution and, consequently, in the expected revenue. Nonetheless, the qualitative shape of the revenue curve remains similar to the exponential case.  As discussed above, high prices reduce the effective arrival rate and therefore the effective sample size, making  large samples essential for reliable estimation.  Importantly, Figure~\ref{fig:expected_revenue_hyperexponential} also demonstrates that a price level that is ``low''  for one parameter configuration may be ``high'' for another. 
For these reasons, we set \(k_1=10^4\), \(p_1=1\), and \(g(k_i)=2k_i\), ensuring rapid growth of the sample  size across iterations.

\begin{figure}[H]
{\centering \includegraphics[width=0.7\linewidth]{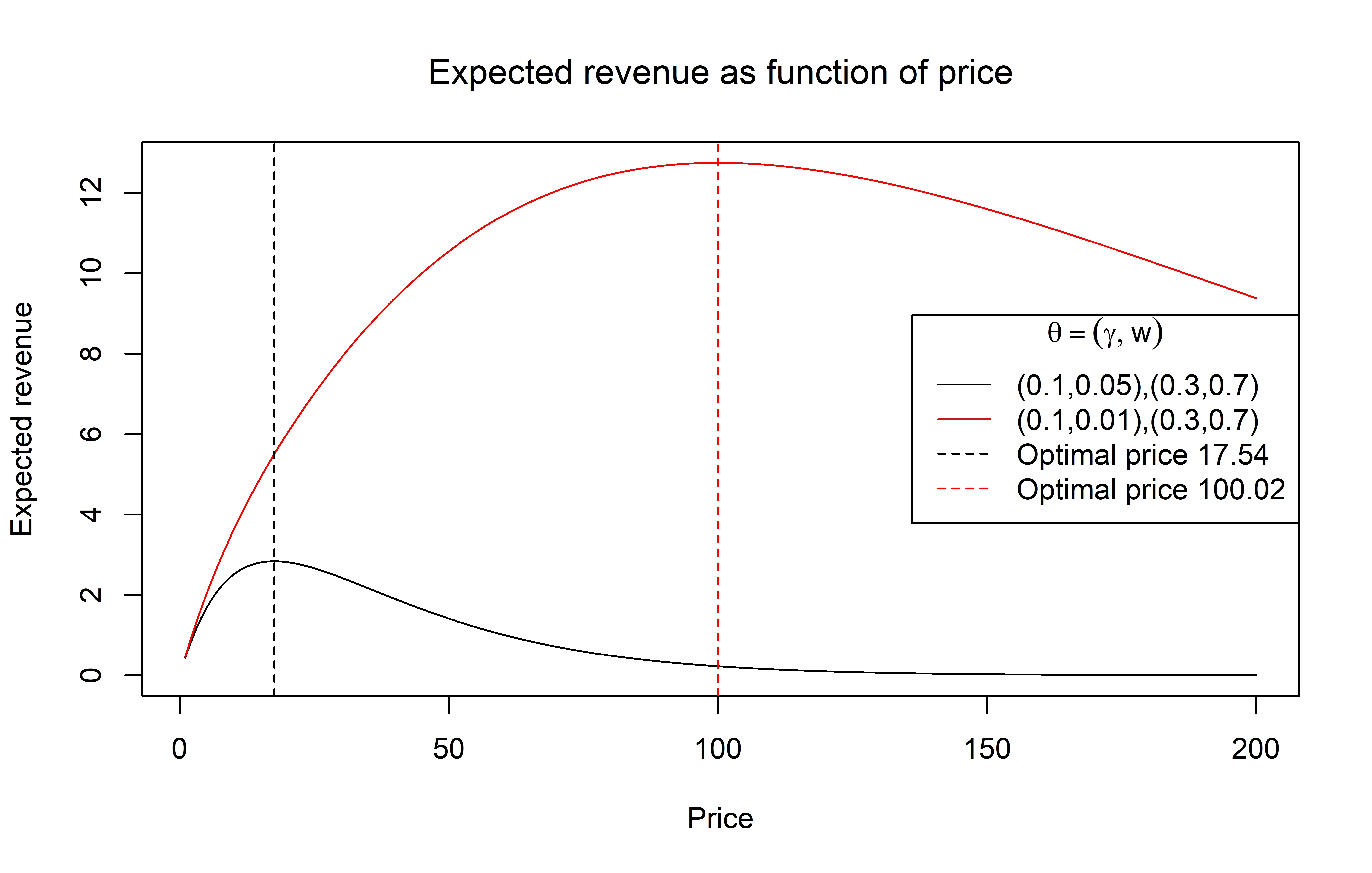}
}
\caption{Expected revenue as function of price. Simulation parameters: $\lambda=0.5,\: \mu=1,\: C=1,\: p=5$. }\label{fig:expected_revenue_hyperexponential}
\end{figure}

Figure~\ref{fig:price_convergence_hypexponential} demonstrates that for the case 
\(\pmb{\gamma}=\{0.1,0.05\}\), the achieved revenue exceeds \(99\%\) of the optimal value after 
only three iterations, corresponding to a total of approximately \(70{,}000\) observations. 
This stands in sharp contrast to the direct estimation of \(\pmb{\theta}\) in Section~4.2, 
which required roughly \(400{,}000\) observations to attain satisfactory accuracy. 
For the more challenging configuration \(\pmb{\gamma}=\{0.1,0.01\}\), the results are even 
more striking: after five iterations (about \(310{,}000\) observations in total), the revenue already 
exceeds \(99\%\) of the optimum, whereas reliable estimation of \(\pmb{\theta}\) in this case 
required at least \(10^{7}\) observations and was therefore omitted from further analysis.

\begin{figure}[H]
{\centering \includegraphics[width=0.7\linewidth]{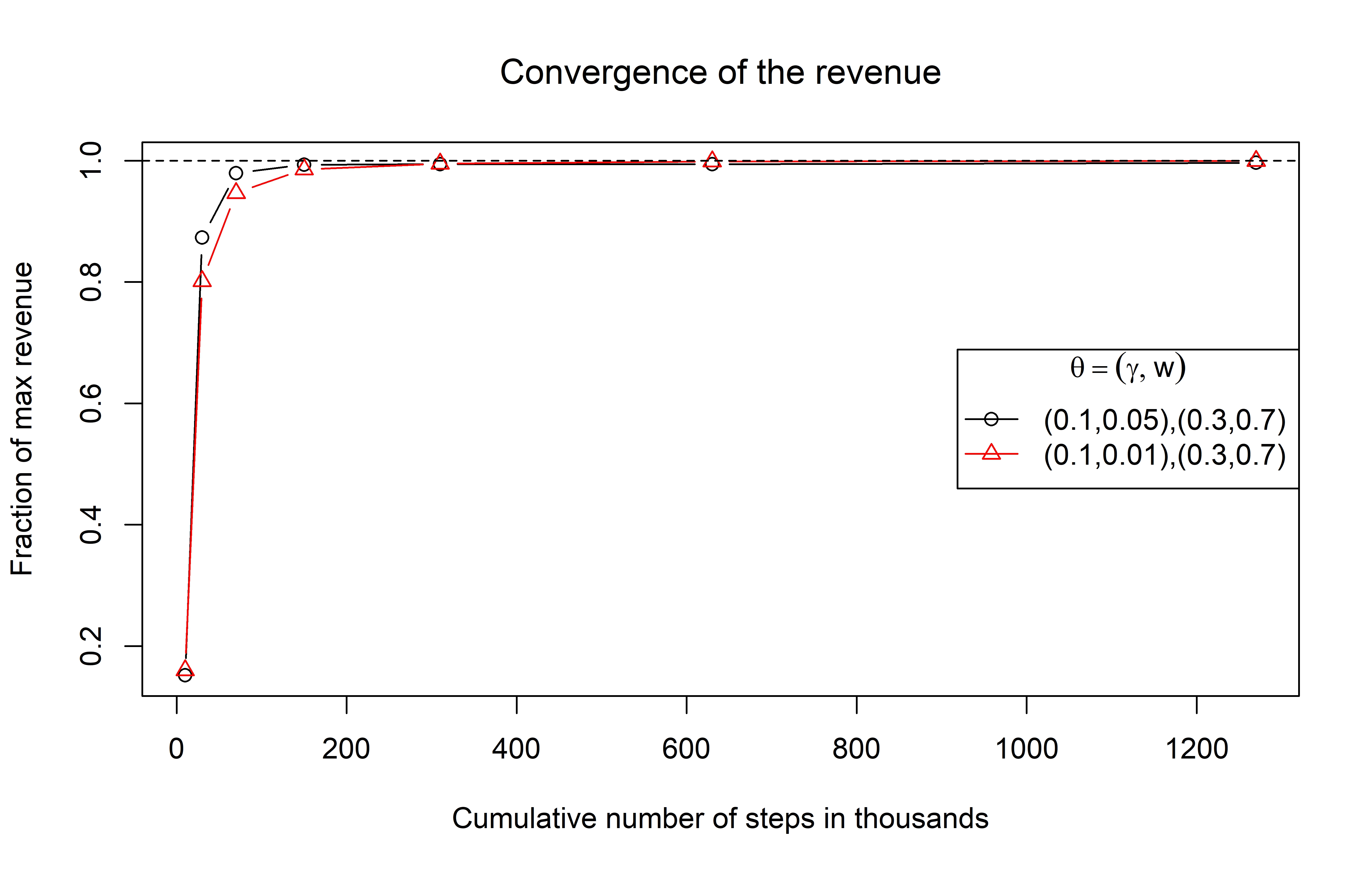}
}
\caption{Expected revenue as function of price (average of 100 independent runs for each combination of parameters). Simulation parameters: $\lambda=0.5,\: \mu=1,\: \pmb{\gamma}=\{0.1,\:0.05\},\: \pmb{w}=\{0.3,\: 0.7\},\: C=1,\: p=5$.}\label{fig:revenue_convergence_hypexponential}
\end{figure}

Figure~\ref{fig:revenue_convergence_hypexponential} shows that price convergence is noticeably slower 
than revenue convergence, particularly for the case \(\pmb{\gamma}=\{0.1,0.01\}\). 
This behavior can be explained by two factors. 
First, the expected revenue function is relatively flat in a neighborhood of its maximum 
(see Figure~\ref{fig:expected_revenue_hyperexponential}), so substantial deviations in price may lead 
to only marginal changes in revenue. 
Second, for the parameter configuration \(\pmb{\gamma}=\{0.1,0.05\}\), the expected revenue is 
generally less sensitive to price changes across a wide range of prices, further decoupling price 
convergence from revenue convergence.

\begin{figure}[H]
\centering \includegraphics[width=0.7\linewidth]{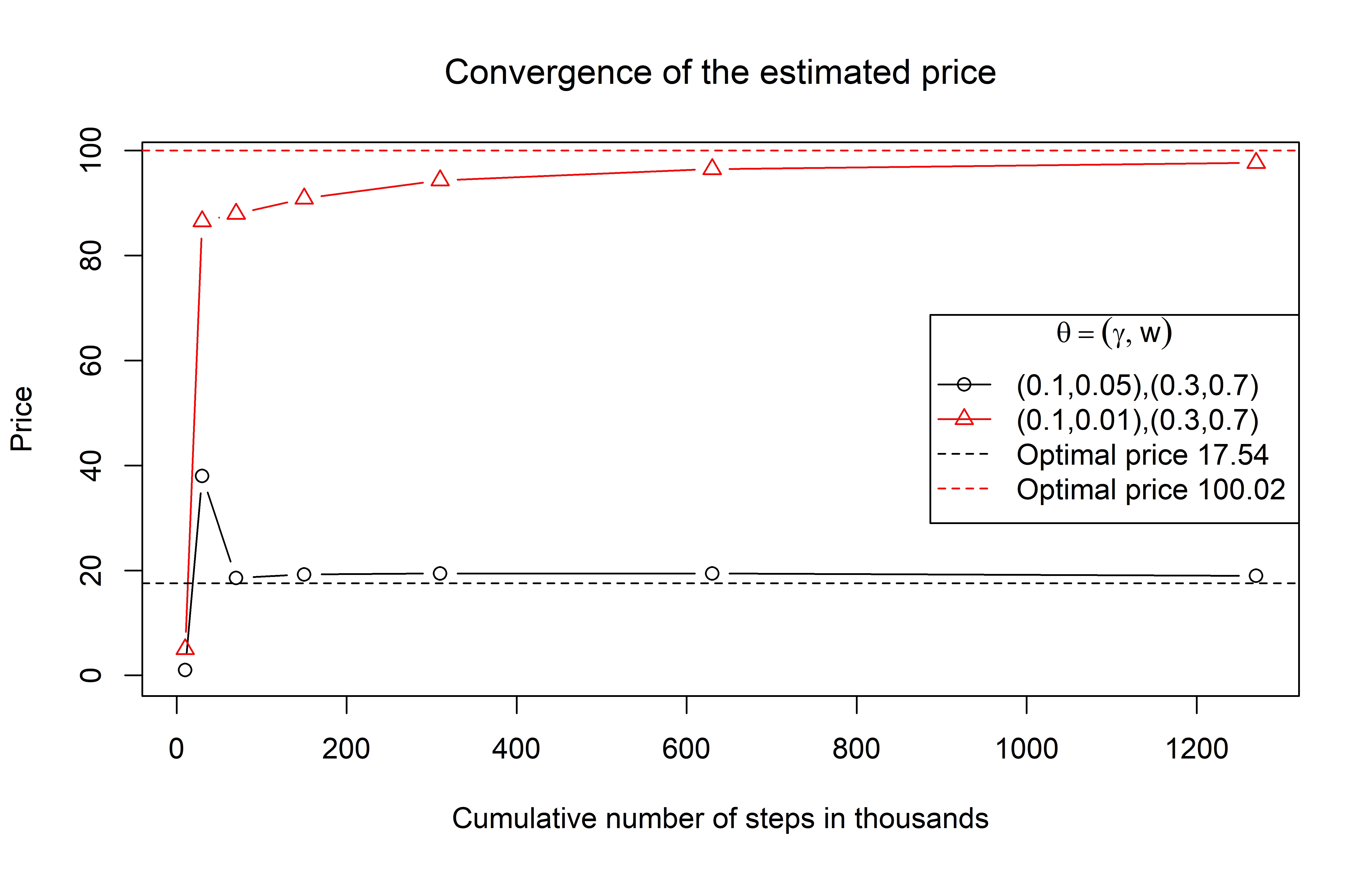}
\caption{Expected revenue as function of price (average of 100 independent runs for each combination of parameters). Last estimated prices are $18.2$ and $98.7$.} Simulation parameters: $\lambda=0.5,\: \mu=1,\: \pmb{\gamma}=\{0.1,\:0.05\},\: \pmb{w}=\{0.3,\: 0.7\},\: C=1,\: p=5$. \label{fig:price_convergence_hypexponential}
\end{figure}

We conclude that the proposed iterative pricing algorithm performs well also in the 
hyperexponential setting, even though parameter estimation in this case requires 
substantially more data than in the exponential model. As the hyperexponential family can be used to approximate a broad class of distributions, this further suggests that specific knowledge of the actual parametric form of the true service value distribution is not essential for the revenue maximization problem. This is especially true when taking hyperexponential distributions with a higher parametric dimension. Note, however, this comes at a cost, in terms of computational effort and data collection, as explained in Section~\ref{sec:hyperexponential-service-value.}.


\section{Conclusions and discussion.}\label{sec:conclusions-and-extensions.}

In this work, we relaxed the assumption of homogeneous value from completed service ($R$) in Naor's model. We proposed maximum likelihood estimator for an unknown parameter of the known distribution of $R$ based on queue length data. The use of queue data allows to estimate the parameter of distribution of $R$ indirectly, without observing samples of $R$ directly. We showed that, under standard assumptions, the estimator is consistent, and the normalized errors are asymptotically normal around 0 with a covariance matrix that can be computed. These results were demonstrated using simulation experiments, assuming exponential or hyperexponential distributions for $R$. Additionally, an iterative price maximization algorithm based on the estimator was constructed and applied to simulation data to demonstrate the estimator's usefulness.

The main advantage of our parametric framework, in comparison with black-box methods, is that the parameters have a concrete interpretation for the underlying economic model. Moreover, this approach enables sensitivity analysis of the learning and optimization performance with respect to the model and algorithm parameters. 



We emphasize that revenue optimization is only one of possible applications. Pricing algorithm can optimize other target functions. For example, same algorithm can be used to find socially optimal price. For this purpose, revenue function should be replaced with stationary social welfare per unit of time

\[
\mathrm{SW}(p,\pmb{\theta})
=\lambda\sum_{q=0}^{\infty}\pi_{\pmb{\theta}}(q)\int_{r(q)}^{\infty}\bigl(1-F_{\pmb{\theta}}(x)\bigr)\,dx.
\]

where $\pi_{\pmb{\theta}}(q)$ and $\bigl(1-F_{\pmb{\theta}}(x)\bigr)\,dx$ can be approximated via truncation. Algorithm 1 can be applied directly to the social welfare optimization problem, with the only change is in the first order condition.

Alternatively, a similar iterative algorithm can be developed for optimizing the service rate when the  administrator cannot control the admission price but can control service capacity. For example, if the service rate \(\mu\) is a controllable decision variable (e.g., through staffing or resource allocation), one can replace the price update step with an update of \(\mu\) and maximize the chosen stationary objective (revenue, social welfare, or a cost-penalized objective) over \(\mu\), using the current parameter estimates. Since \(\mu\) affects both the delay term \(r(q)=p+(q+1)C/\mu\) and the stationary queue-length distribution, the same estimation--optimization loop applies, with the decision variable changed from \(p\) to \(\mu\).

Another advantage of estimator based on queue length only is simplicity of data collection in real scenario, as data collection of balking customers if often difficult, or even impossible. However, if such data is available, it brings additional information which can be utilized. In this case, the ``balking'' event is added to the jump chain. The probability of such event is $\lambda F(r(q_i))/(F(r(q_i))+\mu)$. Each time such event occur, the MLE is multiplies by the corresponding probability. Such additional information would decrease the variance of the estimator.

The estimation procedure presented here can be taken a step further to a non-parametric approach where for each queue length $q_i$ a different state dependent arrival rate $\lambda_i$ is estimated. This does not require assuming a specific decision making mechanism for the customers; i.e., this can capture heterogeneity in rewards, delay sensitivity or simply noisy information. This is appealing because it is very general but is less interpretable from a modeling perspective. This would also entail new technical difficulties and bias that comes along with with nonparametric inference (estimating an infinite sequence in this case).

A final observation is that it is straightforward to extend our analysis to more general queues, such as M/M/s or M/M/1/r (which still have an underlying birth-death process). Another possibility is to relax another assumption of Naor's model—the homogeneity of customers' patience ($C$). A similar approach can be used to construct an MLE for a common joint distribution of $R$ and $C$. In such a setting, a customer joins the queue if $\frac{R-p}{C}\geq \frac{q_{i}+1}{\mu}$. The distribution of $\frac{R-p}{C}$ can be viewed as one with an additional known parameter $p$. In a special case where $p=0$, our estimator can serve as the MLE for the parameter of the common distribution of $\frac{R}{C}$. For $p>0$, estimating the joint distribution of $R$ and $C$ for a fixed price becomes less straightforward. One possibility is to construct a two-step procedure that first estimates the parameters for different prices, then seeks an optimal solution in the second step.

A similar approach is often found in information technology, where a new service is initially offered for free to attract as many users as possible and study their characteristics. Later, the service becomes paid. The proposed approach assumes that the cost to a customer for staying in the system per unit of time is a linear function of the expected service time $\frac{q_{i}+1}{\mu}$, with service times exponentially distributed. Once these assumptions are removed and the cost to a customer for staying in the system per unit of time is considered as some distribution of expected waiting time, $C(w(q_i))$, more general models can be constructed.



\pagebreak

\begin{appendices}
\section{Proofs}\label{appendixes}

First we introduce a few expressions frequently used throughout the
proofs. Recall that the gradient of the log-likelihood is given by
\[\nabla\log \mathcal{L}(\pmb{\theta}|Q_k,Y_k):=\pmb{\Psi}_k(\pmb{\theta})=\frac{1}{k}\sum_{i\in M}\pmb{\psi}(Q_{i-1},Y_i,\pmb{\theta})\]
where,   for any $q\geq 0$ and $y\in\{0,1\}$,
\begin{align}\label{eq:psi}
\pmb{\psi}(q,y,\pmb{\theta})=
\begin{bmatrix}
y \frac{p'_{\theta_1}(q,\pmb{\theta})}{p(q,\pmb{\theta})}
- (1-y) \frac{p'_{\theta_1}(q,\pmb{\theta})}{1-p(q,\pmb{\theta})}
\\
\\
\vdots
\\
\\
y \frac{p'_{\theta_n}(q,\pmb{\theta})}{p(q,\pmb{\theta})}
- (1-y) \frac{p'_{\theta_n}(q,\pmb{\theta})}{1-p(q,\pmb{\theta})}
\end{bmatrix}
\end{align}
The upward transition probability (given $\pmb{\theta}$) from state $q$ is
\begin{align}\label{eq:defp}
p(q,\pmb{\theta})=\frac{\lambda(1-F(r(q)),\pmb{\theta})))}{\mu+\lambda(1-F(r(q),\pmb{\theta}))},
\end{align}
where $r(q)=p+\frac{(q+1)C}{\mu}$.  For any $j=1,\ldots,n$, taking derivative with respect to $\theta_j$ yields
\begin{align}\label{eq:derp}
p'_{\theta_j}(q,\pmb{\theta}):=\frac{\partial p(q,\pmb{\theta})}{\partial \theta_j}=
-\frac{\mu\lambda F'_{\theta_j}(r(q),\pmb{\theta})}{(\mu+\lambda(1-F(r(q),\pmb{\theta})^2},
\end{align}
and once more with respect to $\theta_l$, for any $l=1,\ldots,n$,
\begin{equation}\label{eq:derp2}
p''_{\theta_j \theta_l}(q,\pmb{\theta}) =\frac{-\mu\lambda[F''_{\theta_j\theta_l}(r(q),\pmb{\theta})(\mu+\lambda(1-F(r(q),\pmb{\theta})+2\lambda F'_{\theta_j}(r(q),\pmb{\theta})F'_{\theta_l}(r(q),\pmb{\theta})]}{(\mu+\lambda(1-F(r(q),\pmb{\theta})^3},
\end{equation}
where $p''_{\theta_j \theta_l}(q,\pmb{\theta}) := \frac{\partial p'_{\theta_j}(q,\pmb{\theta})}{\partial \theta_l}$. 

By Assumption~A2, all of the expressions above are well defined for any $(q,y)$ such that $F(r(q),\pmb{\theta})\in(0,1)$, and the first derivative $p'_{\theta_j}(q,\pmb{\theta})$ is continuous w.r.t. $\pmb{\theta}$ for every $j=1,\ldots,n$. Note that for our results the second derivative need not be continuous, i.e., $F''$ exists but may have discontinuities.

\subsection{Consistency (proofs of Lemmas
1-5)}\label{appendix-1.-proof-of-lemmas-1-5.}

\setcounter{lemma}{0}
\begin{lemma}
If $F(r(q),\pmb{\theta})$ is continuous, twice differentiable and has a continuous first derivative, w.r.t. $\pmb{\theta}\in \pmb{\Theta}$ (Assumption~A2), then  $\pmb{\psi}_i(q,y,\pmb{\theta})$ is continuous  w.r.t. $\pmb{\theta}\in \pmb{\Theta}$, for any $(q,y)$ such that $F(r(q),\pmb{\theta})\in(0,1)$.
\end{lemma}

\begin{proof} Observe that the denominator of \eqref{eq:derp} is strictly positive for any $q$ such that $F(r(q),\pmb{\theta})\in(0,1)$.  Therefore, as $F$ and its first derivative is continuous with respect to $\pmb{\theta}\in\pmb{\Theta}$,  we conclude from \eqref{eq:psi} and \eqref{eq:defp} that $\pmb{\psi}$ is also a continuous function. 

\end{proof}

\begin{lemma}
Suppose Assumption A3 holds, then
\begin{align*}
\max_{j=1\ldots,n}\left|\pmb{\Psi}_k(\pmb{\pmb{\theta}})\right|_j\leq 2 H(q)\ , \forall q: F(r(q),\pmb{\theta})\in(0,1),
\end{align*}
 where $H$ is the integrable function in Assumption~A3.
\end{lemma}

\begin{proof}
For any $j=1,\ldots,n$,
\begin{align*}
\psi^j(q,y,\pmb{\theta}) =y\frac{p'_{\theta_j}(q,\pmb{\theta})}{p(q,\pmb{\theta})} +(1-y) \frac{p'_{\theta_j}(q,\pmb{\theta})}{1-p(q,\pmb{\theta})},
\end{align*}
hence, as $y\in\{0,1\}$, the triangle inequality yields
\begin{equation}\label{eq:triangle}
\left|\psi^j(q,y,\pmb{\theta})\right|\leq  \left| \frac{p'_{\theta_j}(q,\pmb{\theta})}{p(q,\pmb{\theta})}\right|+\left|\frac{p'_{\theta_j}(q,\pmb{\theta})}{1-p(q,\pmb{\theta})}\right| ,
\end{equation}
By \eqref{eq:defp} and \eqref{eq:derp},
\begin{equation}\label{eq:lem_bound1}
\left|\frac{p'_{\theta_j}(q,\pmb{\theta})}{1-p(q,\pmb{\theta})}\right|=\left|-\frac{\lambda F'_{\theta_j}(r(q),\pmb{\theta})}{\mu+\lambda(1-F(r(q),\pmb{\theta}))} \right|<\left|\frac{ F'_{\theta_j}(r(q),\pmb{\theta}))}{1-F(r(q),\pmb{\theta})}\right| ,
\end{equation}
and
\[
\frac{p'_{\theta_j}(q,\pmb{\theta})}{p(q,\pmb{\theta})}=
-\frac{\mu F_{\theta_j}'(r(q),\pmb{\theta})}
{(1-F(r(q),\pmb{\theta}))(\mu+\lambda(1-F(r(q),\pmb{\theta})))}=
\frac{-F_{\theta_j}'(r(q),\pmb{\theta})}{1-F(r(q),\pmb{\theta})}\cdot\frac{\mu}{\mu+\lambda(1-F(r(q),\pmb{\theta}))}.
\]
Note that \(0\leq F(r(q),\pmb{\theta}) \leq 1\). As a result
\[
\frac{\mu}{\mu+\lambda}\leq\frac{\mu}{\mu+\lambda(1-F(r(q),\pmb{\theta}))}\leq 1,
\]
thus,
\begin{equation}\label{eq:lem_bound2}
\left|\frac{p'_{\theta_j}(q,\pmb{\theta})}{p(q,\pmb{\theta})}\right|=
\left|\frac{-F_{\theta_j}'(r(q,\pmb{\theta})}{1-F(r(q),\pmb{\theta})}\cdot\frac{\mu}{\mu+\lambda(1-F(r(q),\pmb{\theta}))}\right|\leq \left|\frac{ F'_{\theta_j}(r(q),\pmb{\theta}))}{1-F(r(q),\pmb{\theta})}\right|. 
\end{equation}
Combining the above with \eqref{eq:triangle} and Assumption~A3 we conclude that
\begin{align*}
\max_{j=1,\ldots,n}\left| {\psi}^j(q,y,\pmb{\theta}) \right|_j\leq  2\max_{j=1,\ldots,n}\left|\frac{ \nabla_{\pmb{\theta}} F(r(q),\pmb{\theta})}{1-F(r(q),\pmb{\theta})}\right| \leq 2H(q) .
\end{align*}
The proof is completed by recalling that $\pmb{\Psi}_k(\pmb{\pmb{\theta}})=\frac{1}{k}\sum_{i=1}^k\pmb{\psi}(Q_{i-1},Y_i,\pmb{\theta})$.

\end{proof}

\begin{lemma}
A unique stationary distribution $(Q_0,Y_1)$ exists and for any function $g:\mathbb{N}\times\{0,1\}\to\mathbb{R}$ such that $\mathbb{E}[|g(Q_0,Y_1)|]<\infty$,
$$\frac{1}{k}\sum_{i\in M} g(Q_{i-1},Y_{i})\overset{a.s.}{\longrightarrow}\mathbb{E}[g(Q_0,Y_1)]\:\: \text{as} \:\:k\rightarrow\infty.$$
Further, if Assumption A3\: holds, then for any $\pmb{\theta}\in\pmb{\Theta}$,
$$\pmb{\Psi}_k(\pmb{\theta})\overset{a.s.}{\longrightarrow}\pmb{\Psi}(\pmb{\theta})\:\: \text{as }\: k\rightarrow\infty.$$
\end{lemma}
\begin{proof}
An irreducible birth-death process with rates
\((\lambda_0,\mu_1,\lambda_1,\mu_2,\lambda_2,...,)\) is ergodic and has
a unique stationary distribution if the following conditions hold
(see \cite{karlin_classification_1957}): 
\begin{align}\label{eq:BDCond1}
\sum_{k=1}^{\infty}\prod_{q=1}^{k}\frac{\lambda_{q-1}}{\mu_q}<\infty,
\end{align}

\begin{align}\label{eq:BDCond2}
\sum_{k=1}^{\infty}\prod_{q=1}^{k}\frac{\mu_q}{\lambda_q}=\infty.
\end{align}

Recall that \(r(q)=p+\frac{C(q_{}+1)}{\mu}\). Also, \(\mu_q=\mu\) for
all \(q\). In addition, \(F(r,\pmb{\theta})\) is a CDF, therefore
\(\lim_{r \rightarrow\infty}{F(r,\pmb{\theta})}=1\), consequently

\begin{align}\label{eq:BD1}
\lim_{q_{} \rightarrow\infty}\lambda_i=
\lim_{q_{} \rightarrow\infty}\lambda\left(1-F\left(p+\frac{C(q_{}+1)}{\mu}\right)\right)=0.
\end{align}

We first verify \eqref{eq:BDCond1}. From \eqref{eq:BD1} it follows that \(\lambda_q\rightarrow 0\) as
\(q\rightarrow \infty\). Thus there exists a \(q\) such that
\(\frac{\lambda_{q-1}}{\mu}<1\). Let

\[q_0=\underset{q}{\text{arg min}}\left\{q:\frac{\lambda_{q-1}}{\mu}<1\right\},\]
then we can rewrite
\begin{align}\label{eq:BD2}
\sum_{k=1}^{\infty}\prod_{q=1}^{k}\frac{\lambda_{q-1}}{\mu}=\omega_0+
\sum_{k=q_0}^{\infty}\prod_{q=1}^{k}\frac{\lambda_{q-1}}{\mu},
\end{align}
where \(\omega_0:=\sum_{k=1}^{q_0-1}\prod_{q=1}^{k}\frac{\lambda_{q-1}}{\mu}\) is a positive finite number. Similarly,
\begin{align}\label{eq:BD3}
\prod_{q=1}^{k}\frac{\lambda_{q-1}}{\mu}=\omega_1\prod_{q=q_0}^{k}\frac{\lambda_{q-1}}{\mu},
\end{align}
where \(\omega_1\:=\prod_{q=1}^{q_0-1}\frac{\lambda_{q-1}}{\mu}\) is a finite positive number larger than \(1\). Combining \eqref{eq:BD2} and \eqref{eq:BD3} yields
\begin{align}\label{eq:BD4}
\sum_{k=1}^{\infty}\prod_{q=1}^{k}\frac{\lambda_{q-1}}{\mu}=
\omega_0+\omega_1\sum_{k=q_0}^{\infty}\prod_{q=q_0}^{k}\frac{\lambda_{q-1}}{\mu}.
\end{align}
Now, since \(F(r,\pmb{\theta})\) is monotone increasing function, we
have that for any \(q\),
\(\frac{\lambda_{q}}{\mu}<\frac{\lambda_{q-1}}{\mu}\),  and in particular
\[\frac{\lambda_{q_0+1}}{\mu}<\frac{\lambda_{q_0}}{\mu}.\]
As a result, for any \(k\geq q_0+1\)
\[\prod_{q=q_0}^{k}\frac{\lambda_{q-1}}{\mu}<
\left(\frac{\lambda_{q_0}}{\mu}\right)^{k-q_0+1}.\]
Combining this with \eqref{eq:BD4}, we have that
\[\omega_0+\omega_1\sum_{k=q_0}^{\infty}\prod_{q=q_0}^{k}\frac{\lambda_{q-1}}{\mu}<
\omega_0+\omega_1\sum_{k=q_0}^{\infty}\left(\frac{\lambda_{q_0}}{\mu}\right)^{k-q_0+1}.\]
Since \(\frac{\lambda_{q_0}}{\mu}<1\), it follows that \(\sum_{k=q_0}^{\infty}\left(\frac{\lambda_{q_0}}{\mu}\right)^{k-q_0+1}<\infty\). Thus,  \eqref{eq:BDCond1} is verified.

We next verify \eqref{eq:BDCond2}. As was shown previously, \(\lambda_q \rightarrow 0\) as
\(q \rightarrow \infty\) while \(\mu\) is constant. Thus
\(\frac{\mu_q}{\lambda_q}=\frac{\mu}{\lambda_q}\rightarrow \infty\) when
\(q \rightarrow \infty\),  hence
\[\sum_{k=1}^{\infty}\prod_{i=1}^{k}\frac{\mu}{\lambda_q}=\infty\]
as required. We conclude that \((Q_i,y)_{i\geq 0}\) is an ergodic sequence and a unique stationary distribution \(Q_0,Y_1\) exists.  Moreover, by Ergodic Theorem in \cite[Ch.~6]{durrett_probability_2019}, for any ergodic sequence, and, for any integrable function $g$ there exists a limit
$$\frac{1}{k}\sum_{i\in M} g(Q_{i-1},Y_{i})\overset{a.s.}{\longrightarrow}\mathbb{E}[g(Q_0,Y_1)] <\infty .$$
Further, if Assumption A3\: holds, then Lemma~\ref{lem:bound} implies that $\pmb{\Psi}_k(\pmb{\theta})\overset{a.s.}{\longrightarrow}\pmb{\Psi}(\pmb{\theta})$ for any $\pmb{\theta}\in\pmb{\Theta}$.
\end{proof}

\begin{lemma}
If Assumptions A3 holds, then for any $\epsilon>0$.

$$\inf_{\pmb{\theta}:d(\pmb{\theta},\pmb{\theta}_0)\geq \epsilon}||\pmb{\Psi}(\pmb{\theta})||>0=||\pmb{\Psi}(\pmb{\theta}_0)||.$$
\end{lemma}
\begin{proof}
Applying Lemma~\ref{lem:ergodic} together with \eqref{eq:lem_bound1}, for any $\pmb{\theta}$ there exists a limit as $k\to\infty$,
\[
\frac{1}{k}\sum_{i\in M}\left[Y_i \frac{p'_{\theta_j}(Q_{i-1},\pmb{\theta})}{p(Q_{i-1},\pmb{\theta})}\right]\overset{a.s.}{\longrightarrow}
\mathbb{E}\left[ Y_1 \frac{p'_{\theta_j}(Q_0,\pmb{\theta})}{p(Q_0,\pmb{\theta})}\right] <\infty.
\]
Now, using the law of total expectation
\[\mathbb{E}\left[ Y_1 \frac{p'_{\theta_j}(Q_0,\pmb{\theta})}{p(Q_0,\pmb{\theta})}\right]=
\mathbb{E}\left[\mathbb{E}\left[ Y_1 \frac{p'_{\theta_j}(Q_0,\pmb{\theta})}{p(Q_0,\pmb{\theta})}|Q_0\right]\right].\]
Observe that given \(Q_0\) the distribution of \(Y_1\) is Bernoulli with probability $p(Q_0,\pmb{\theta}_0)$, therefore 
\[
\mathbb{E}\left[\mathbb{E}\left[ Y_1 \frac{p'_{\theta_j}(Q_0,\pmb{\theta})}{p(Q_0,\pmb{\theta})}\right]|Q_0\right]=
\mathbb{E}\left[\frac{p'_{\theta_j}(Q_0,\pmb{\theta})}{p(Q_0,\pmb{\theta})}\mathbb{E}\left[ Y_1|Q_0 \right]\right]=
\mathbb{E}\left[\frac{p'_{\theta_j}(Q_0,\pmb{\theta})}{p(Q_0,\pmb{\theta})}p(Q_0,\pmb{\theta}_0)\right].
\]
Only for \(\pmb{\theta}=\pmb{\theta}_0\) we obtain
\begin{align}\label{eq:thm3.1}
\mathbb{E}\left[\mathbb{E}\left[ Y_1 \frac{p'_{\theta_j}(Q_0,\pmb{\theta}_0)}{p(Q_0,\pmb{\theta}_0)}\right]|Q_0\right]=\mathbb{E}\left[p'_{\theta_j}(Q_0,\pmb{\theta}_0)\right]
\end{align}
We will use similar arguments for the second part of the estimating
equation. Specifically, by Lemma~\ref{lem:ergodic} and \eqref{eq:lem_bound2},
\[
\frac{1}{k}\sum_{i\in M}\left[ (1-y) \frac{p'_{\theta_j}(Q_{i-1},\pmb{\theta})}{1-p(Q_{i-1},\pmb{\theta})}\right]\overset{a.s.}{\longrightarrow}
\mathbb{E}\left[ (1-Y_1) \frac{p'_{\theta_j}(Q_0,\pmb{\theta})}{1-p(Q_0,\pmb{\theta})}\right]<\infty .
\]

Applying the law of total expectation yields,
\begin{align*}
\mathbb{E}\left[ (1-Y_1) \frac{p'_{\theta_j}(Q_0,\pmb{\theta})}{1-p(Q_0,\pmb{\theta})}\right] &=
\mathbb{E}\left[\mathbb{E}\left[ (1-Y_1) \frac{p'_{\theta_j}(Q_0,\pmb{\theta})}{1-p(Q_0,\pmb{\theta})}|Q_0\right]\right] 
=\mathbb{E}\left[\frac{p'_{\theta_j}(Q_0,\pmb{\theta})}{1-p(Q_0,\pmb{\theta})}\mathbb{E}\left[1- Y_1|Q_0 \right]\right] \\
&=
\mathbb{E}\left[\frac{p'_{\theta_j}(Q_0,\pmb{\theta})}{1-p(Q_0,\pmb{\theta})}(1-p(Q_0|\pmb{\theta}_0))\right] .
\end{align*}
Plugging in \(\pmb{\theta}=\pmb{\theta}_0\) yields
\begin{align}\label{eq:thm3.2}
\mathbb{E}\left[ (1-Y_1) \frac{p'_{\theta_j}(Q_0,\pmb{\theta}_0)}{1-p(Q_0,\pmb{\theta}_0)}\right]=\mathbb{E}\left[p'_{\theta_j}(Q_0,\pmb{\theta}_0)\right].
\end{align}
Then combination of \eqref{eq:thm3.1},  \eqref{eq:thm3.2} and \eqref{eq:psi}
completes the proof of Lemma~\ref{lem:unique}.
\end{proof}

\begin{lemma}
If Assumptions A1, A2, A3 hold, then
$$\pmb{\Psi}_k(\hat{\pmb{\theta}}_k)\overset{a.s.}{\longrightarrow}\pmb{0}\:$$
and consequently
$$\pmb{\Psi}_k(\hat{\pmb{\theta}}_k)\overset{P}{\longrightarrow}\pmb{0}\:.$$
\end{lemma}

\begin{proof}
Lemmas~\ref{lem:continuity}-\ref{lem:ergodic} imply that $\pmb{\Psi}_k(\hat{\pmb{\theta}})$ converges uniformly on the compact set $\pmb{\Theta}$ to the function $\pmb{\Psi}(\hat{\pmb{\theta}})$ on $\pmb{\Theta}$ (see \cite[Thm.~16a]{book_F1996}). By Lemma~\ref{lem:unique} and Assumption~A1, $\pmb{\Psi}(\pmb{\theta})$) has a unique root $\pmb{\theta}_0\in\pmb{\Theta}^{\mathrm{o}}$ (the interior of the compact parameter space). Therefore, there exists a constant $\delta>0$ such that
\begin{align*}
    \inf_{\pmb{\theta}\notin\pmb{\Theta}^{\mathrm{o}}} |\pmb{\Psi}(\hat{\pmb{\theta}})|_j\geq \delta , \forall j=1,\ldots,n.
\end{align*}
Moreover, as $\pmb{\Theta}^{\mathrm{o}}\subset\pmb{\Theta}$ \eqref{eq:uniform}, implies that 
\begin{align*}
\inf_{\pmb{\theta}\notin\pmb{\Theta}^{\mathrm{o}}}\min_{j=1,\ldots,n}|\pmb{\Psi}_k(\hat{\pmb{\theta}})| \overset{\mathrm{a.s.}}{\longrightarrow}\delta >0.
\end{align*}
Thus, with probability one, there exists an integer \(N\), such that all coordinates of $\inf_{\pmb{\theta}\notin\pmb{\Theta}^{\mathrm{o}}}\pmb{\Psi}(\hat{\pmb{\theta}})$ are non-zero for any
  \(k>N\).  I.e., the MLE is given by an interior solution for any $k>N$.    We conclude that  \(\pmb{\Psi}_k(\hat{\pmb{\theta}}_k)\overset{a.s.}{\longrightarrow}0\) which also  implies \(\pmb{\Psi}_k(\hat{\pmb{\theta}}_k)\overset{p}{\longrightarrow}0\) and
  completes the proof.
\end{proof}

\subsection{Filtration of the queue-length random walk.}\label{appendix-filtration}

Previously we defined the random process \(Q_i\) in a natural and
intuitive way. Here we define \(Q_i\) in a formal way. The rigorous
definition is useful for some proofs and enables a deeper understanding
from a probabilistic angle. Firstly we define the probability space
\((\Omega_k,\mathcal{F}_i,\mathrm{P}_i)\) for \(Q_i\). Let \(\Omega_k\)
be the sample space of a queue process after \(k\) steps. \(\Omega_k\)
includes all possible paths which are sets of length \(k+1\) of the form
\(\{ 0,q_1,q_2,...,q_k \}\). Such that \(q_1=1\) and
\(q_i\in \{q_{i-1}-1,q_{i-1}+1\}\) for \({2\leq i \leq k}\).

Let \(\mathcal{F}_i\) be a \(\sigma\)-algebra of step \(i\). Define

\[\phi_m^{i-1}=\left\{ \bigcup_{j\in \Gamma_m}\omega_j \right\},\]

where

\[\Gamma_{m}=\left\{ j: w_j \in \{ 0,q_1,q_2,...,q_{i-1}, \cdot \}_m \right\},\]

where \(0,q_1,q_2,...,q_{i-1}\) is a specific path \(m\) and
``\(\cdot\)'' is some unknown path. Also, \(m\in M_{i-1}\) where
\(M_{i-1}\) is a group of all possible paths \(0,q_1,q_2,...,q_{i-1}\).

Then \(\mathcal{F}_i\) is a set of all \(m\in M_{i-1}\) sets
\(\phi_m^{i-1}\) and their unions in every possible combination, and
\(\emptyset\).

Also, every \(\phi_m^{i}\) is
\(\left\{ \bigcup_{j\in \Gamma_m}\omega_j \right\}\), \(m\in M_{i}\)
where \(M_{i}\) is a group of all possible paths
\(0,q_1,q_2,...,q_{i}\). Then for each \(m\in M_{i}\), \(\phi_m^{i}\) is
the union of \(w\) of the form
\(\{ 0,q_1,q_2,...,q_{i-1},q_{i-1}+1,\cdot \}\) or
\(\{ 0,q_1,q_2,...,q_{i-1},q_{i-1}-1,\cdot \}\). \(\mathcal{F}_i\)
include all sets \(\phi_m^{i}\) as well as their unions in every
possible combination and a \(\emptyset\).

However, for every \(m\in M_{i-1}\), \(\phi_m^{i-1}\) is a union of
\(\phi_m^{i}\) of the form
\(\{ 0,q_1,q_2,...,q_{i-1},q_{i-1}+1,\cdot \}\) with \(\phi_m^{i}\) of
the form \(\{ 0,q_1,q_2,...,q_{i-1},q_{i-1}-1,\cdot \}\). Then every
\(\phi_m^{i-1}\in \mathcal{F}_i\), as well as any union of
\(\phi_m^{i-1}\), where \(m\in M_{i-1}\) is a subset of
\(\mathcal{F}_i\). As a result,
\(\mathcal{F}_{i-1}\subseteq\mathcal{F}_i\), which allows us to conclude
that \(\{\mathcal{F}_i\}_{i\geq0}\) is filtration.

Now we define a sequence of random variables
\((Q_i)_{0\leq i\leq k}:(\Omega_k,\mathcal{F}_i,\mathrm{P}_i)\rightarrow(\mathbb{Z}^+,\{0,1,...,i\},F_i)\)
as \(i\)-th element of \(\omega\), i.e.~\(Q_i(\omega)=q_i\). The
sequence of random variables \(Q_i\) is adapted to the filtration
\(\mathcal{F}_i\).

The experiment is designed in such a way that the queue length increases
if a new customer joins earlier than the currently served person leaves
(since the service is finished) and decreases if the opposite occurs.
Thus the queue length naturally satisfies the Markov property. It
follows that \((Q_i)_{0\leq i\leq k}\) is discrete-time Markov chain
with transition matrix

\[\Pi_{j,q}=\left\{\begin{matrix}
\left\{\begin{matrix}
1 &,\:\: q>j  \\
0 &,\:\: q<j
\end{matrix}\right. &,\:\: j=0  \\
\left\{\begin{matrix}
\frac{\lambda_j}{\mu+\lambda_j} &,\:\:  q>j \\
\frac{\mu}{\mu+\lambda_j} &,\:\:  q<j
\end{matrix}\right. &,\:\: j>0
\end{matrix}\right.\]

Assuming that \(q_0=0\) the probability mass function of \(Q_i\) is
\((\Pi^i)_0\) where index \(0\) is the first row of the matrix
\((\Pi^i)\). The probability \(\mathrm{P}(Q_i=q)=(\Pi^i)_{0,q}\) and
cumulative distribution function is

\[F_{i}(q)=\sum_{j=0}^q(\Pi^i)_{0,j}\]

The probability of \(\phi_m^{i-1}\) is the probability to see a path
\(q_i,q_{i-1},q_{i-2},...,q_{0}\), which, using the Markov property, is
:

\[\mathrm{P}(\phi_m^{i-1})=\mathrm{P}(Q_i=q_i,Q_{i-1}=q_{i-1},Q_{i-2}=q_{i-2},...,Q_{0}=q_{0})=
\prod_{i=1}^kP(Q_{i}=q_{i}|Q_{i-1}=q_{i-1})=\prod_{i=1}^k\Pi_{i-1,i} \:.\]

Each group in \(\mathcal{F}_i\) is a union of \(\phi_m^{i-1}\),
i.e.~\(\cup_{m\in L}\left\{\phi_m^{i-1}\right\}\) where \(L\) is some
set of \(m\in M_{i-1}\). Thus:

\[
\mathrm{P}(\cup_{m\in L}\left\{\phi_m^{i-1}\right\})=\sum_{m\in L} \mathrm{P}(\phi_m^{i-1}) \:.
\]
Finally,  \(\pmb{\psi}_i(Q_{i-1},Y_i,\pmb{\theta})\) as defined in \eqref{eq:psi} is almost everywhere continuous, hence it is measurable and the sequence $(\pmb{\psi}_i(Q_{i-1},Y_i,\pmb{\theta}))_{i\geq 1}$ is adapted to the filtration $\mathcal{F}_i$.

\subsection{Proof of Lemma~\ref{lem:var}}\label{appendix-normal}

\begin{lemma}
If Assumptions  A1-A4 hold, then
\begin{align*}
\pmb{\nabla}\pmb{\Psi}_k^{-1}(\pmb{\theta}_0)\overset{\mathrm{a.s.}}{\longrightarrow}\pmb{\nabla}\pmb{\Psi}^{-1}(\pmb{\theta}_0)\ ,
\end{align*}
where
\begin{align*}
-\pmb{\nabla}\pmb{\Psi}(\pmb{\theta}_0)=\mathbb{E}[\pmb{\psi}(Q_0,Y_1,\pmb{\theta}_0)\pmb{\psi}(Q_0,Y_1,\pmb{\theta}_0)^T]=\pmb{\Sigma}.
\end{align*}
\end{lemma}

\begin{proof}
Recall that $\pmb{\Psi}_k(\pmb{\theta})=\frac{1}{k}\sum_{i=1}^k\pmb{\psi}(Q_{i-1},Y_i,\pmb{\theta})$.  We will first compute the ergodic limit of the Jacobian at $\pmb{\theta}_0$ (assuming, for now, that it exists),
\begin{align*}
(\nabla\pmb{\Psi}_k(\pmb{\theta}_0))_{jl}=\left(\frac{1}{k}\sum_{i=1}^k\frac{\partial}{\partial \theta_l}{\psi}^j(Q_{i-1},Y_i,\pmb{\theta})\right)_{jl} .
\end{align*}
By \eqref{eq:psi}, 
\begin{align*}
\frac{\partial}{\partial \theta_l}{\psi}^j(q,y,,\pmb{\theta})=\frac{\partial}{\partial \theta_l}\left( y \frac{p'_{\theta_j}(q,\pmb{\theta})}{p(q,\pmb{\theta})}\right)-\frac{\partial}{\partial \theta_l}\left( (1-y) \frac{p'_{\theta_j}(q,\pmb{\theta})}{1-p(q,\pmb{\theta})}\right).
\end{align*}
\[
\frac{\partial}{\partial \theta_l}\left( y \frac{p'_{\theta_j}(q,\pmb{\theta})}{p(q,\pmb{\theta})}\right)=
y\frac{p''_{\theta_j \theta_l}(q,\pmb{\theta})p(q,\pmb{\theta})-p'_{\theta_l}(q,\pmb{\theta})p'_{\theta_j}(q,\pmb{\theta})}{(p(q,\pmb{\theta}))^2},
\]
where the second derivative is well defined due to Assumption~A2.
For \(\pmb{\theta}=\pmb{\theta}_0\), the stationary mean,  i.e., with respect to the distribution of $(Q_{0},Y_1)$,  is computed using the law of total expectation,
\begin{align*}
\mathbb{E}\left[  Y_1\frac{p''_{\theta_j \theta_l}(Q_{0},\pmb{\theta}_0)p(Q_{0},\pmb{\theta}_0)-p'_{\theta_l}(Q_{0},\pmb{\theta}_0)p'_{\theta_j}(Q_{0},\pmb{\theta}_0)}{(p(Q_{0},\pmb{\theta}_0))^2} \right] 
=  \mathbb{E}\left[\frac{p''_{\theta_j \theta_l}(Q_{0},\pmb{\theta}_0)p(Q_{0},\pmb{\theta}_0)-p'_{\theta_l}(Q_{0},\pmb{\theta}_0)p'_{\theta_j}(Q_{0},\pmb{\theta}_0)}{p(Q_{0},\pmb{\theta}_0)} \right].
\end{align*}
Similarly, the derivative of the second part is 
\[
\frac{\partial}{\partial \theta_l}\left( (1-y) \frac{p'_{\theta_j}(q,\pmb{\theta})}{1-p(q,\pmb{\theta})}\right)=
(1-y)\frac{p''_{\theta_j \theta_l}(q,\pmb{\theta})(1-p(q,\pmb{\theta}))+p'_{\theta_l}(q,\pmb{\theta})p'_{\theta_j}(q,\pmb{\theta})}{(1-p(q,\pmb{\theta}))^2},
\]
Again, for \(\pmb{\theta}=\pmb{\theta}_0\), the stationary mean,  i.e., with respect to the distribution of $(Q_{0},Y_1)$, the law of total expectation again yields
\begin{align*}
\mathbb{E}\left[  (1-Y_1)\frac{p''_{\theta_j \theta_l}(Q_{0},\pmb{\theta}_0)(1-p(Q_{0},\pmb{\theta}_0))+p'_{\theta_l}(Q_{0},\pmb{\theta}_0)p'_{\theta_j}(Q_{0},\pmb{\theta}_0)}{(1-p(Q_{0},\pmb{\theta}_0))^2} \right] \\
=\mathbb{E}\left[ \frac{p''_{\theta_j \theta_l}(Q_{0},\pmb{\theta}_0)(1-p(Q_{0},\pmb{\theta}_0))+p'_{\theta_l}(Q_{0},\pmb{\theta}_0)p'_{\theta_j}(Q_{0},\pmb{\theta}_0)}{1-p(Q_{0},\pmb{\theta}_0)} \right]
\end{align*}
Observe that at $\theta=\theta_0$, for every $j,l$,
\begin{align*}
    \mathbb{E}\left[\frac{p''_{\theta_j \theta_l}(Q_{0},\pmb{\theta}_0)(1-p(Q_{0},\pmb{\theta}_0))}{1-p(Q_{0},\pmb{\theta}_0)}\right]=   \mathbb{E}\left[\frac{p''_{\theta_j \theta_l}(Q_{0},\pmb{\theta}_0)p(Q_{0},\pmb{\theta}_0)}{p(Q_{0},\pmb{\theta}_0)}\right]=\mathbb{E}\left[p''_{\theta_j \theta_l}(Q_{0},\pmb{\theta}_0)\right].
\end{align*}
As expectation is linear, combining the derivatives of the two parts yields
\begin{align*}
\mathbb{E}\left[\frac{\partial}{\partial \theta_l}{\psi}^j(Q_{0},Y_1,\pmb{\theta}_0)\right] =
\mathbb{E}\left[ \frac{-p'_{\theta_l}(Q_{0},\pmb{\theta}_0)p'_{\theta_j}(Q_{0},\pmb{\theta}_0)}{(1-p(Q_{0},\pmb{\theta}_0))p(Q_{0},\pmb{\theta}_0)}\right].
\end{align*}
For any $j,l$, plugging in the expressions in \eqref{eq:defp} and \eqref{eq:derp} yields
\begin{align*}
    \mathbb{E}\left[\frac{\partial}{\partial \theta_l}{\psi}^j(Q_{0},Y_1,\pmb{\theta}_0)\right]=-\mathbb{E}\left[ \frac{\lambda\mu F'_{\theta_j}(r(Q_0),\pmb{\theta}_0)F'_{\theta_l}(r(Q_0),\pmb{\theta}_0)}{(1-F(r(Q_0),\pmb{\theta}_0))(\mu+\lambda(1-F(r(Q_0),\pmb{\theta}_0)))^2}\right]=-\Sigma_{jl}.
\end{align*}
By Assumption A4 we have that $|\Sigma_{jl}|<\infty$, thus Lemma~\ref{lem:ergodic} implies  that for any $  j,l\in\{1,\ldots,n\}$
 \begin{align*}
(\nabla\pmb{\Psi}_k(\pmb{\theta}_0))_{jl}=\left(\frac{1}{k}\sum_{i=1}^k\frac{\partial}{\partial \theta_l}{\psi}^j(Q_{i-1},Y_i,\pmb{\theta})\right)_{jl}\overset{\mathrm{a.s.}}{\longrightarrow} -\Sigma_{jl},
\end{align*}
where 
$|\Sigma_{jl}|<\infty$ for all $j,l$.

The last remaining step is computing the stationary variance of the martingale sequence.
As $\mathbb{E}\left[{\psi}^j(Q_{0},Y_1,\pmb{\theta}_0)\right] =0$ for all $j=1,\ldots,n$, the variance of the martingale difference is the matrix
\begin{align*}
\mathbb{E}\left[\pmb{\psi}(Q_{0},Y_1,\pmb{\theta}_0)\pmb{\psi}(Q_{0},Y_1,\pmb{\theta}_0)^T\right].
\end{align*}
The entry at coordinate $j,l$ is given by 
\begin{align*}
\mathbb{E}\left[Y_1^2 \frac{p'_{\theta_j}(Q_{0},\pmb{\theta})p'_{\theta_l}(Q_{0},\pmb{\theta})}{p^2(Q_{0},\pmb{\theta})}-2Y_1(1-Y_1)\frac{p'_{\theta_j}(Q_{0},\pmb{\theta})p'_{\theta_l}(Q_{0},\pmb{\theta})}{p(Q_{0},\pmb{\theta})(1-p(Q_{0},\pmb{\theta}))}+(1-Y_1)^2\frac{p'_{\theta_j}(Q_{0},\pmb{\theta})p'_{\theta_l}(Q_{0},\pmb{\theta})}{(1-p(Q_{0},\pmb{\theta}))^2}\right],
\end{align*}
and as $Y_1(1-Y_1)=0$, applying the law of total expectation once more yields
\begin{align*}
\left(\mathbb{E}\left[\pmb{\psi}(Q_{0},Y_1,\pmb{\theta}_0)\pmb{\psi}(Q_{0},Y_1,\pmb{\theta}_0)^T\right]\right)_{jl} &= \mathbb{E}\left[  \frac{p'_{\theta_j}(Q_{0},\pmb{\theta}_0)p'_{\theta_l}(Q_{0},\pmb{\theta}_0)}{p(Q_{0},\pmb{\theta}_0)}
+\frac{p'_{\theta_j}(Q_{0},\pmb{\theta}_0)p'_{\theta_l}(Q_{0},\pmb{\theta}_0)}{(1-p(Q_{0},\pmb{\theta}_0))} \right]\\
&=\mathbb{E}\left[  \frac{p'_{\theta_j}(Q_{0},\pmb{\theta}_0)p'_{\theta_l}(Q_{0},\pmb{\theta}_0)}{p(Q_{0},\pmb{\theta}_0)(1-p(Q_{0},\pmb{\theta}_0))}\right] =\Sigma_{jl} .
\end{align*}

\end{proof}

\end{appendices}


\begin{thebibliography}{9}

\bibitem{afeche_bayesian_2013}
\newblock{P.~Af\'eche a d B.~Ata (2013).}
\newblock{Bayesian Dynamic Pricing in Queueing Systems with Unknown Delay Cost Characteristics.}
\newblock{\em Manufacturing \& Service Operations Management} 15 (2):292--304.

\bibitem{AAES2013}
{Z.~Ak\c{s}in, B.~Ata, S.M.~Emadi, C.L.~Su} (2013).
Structural estimation of callers' delay sensitivity in call centers. {\em Management Science}, 59:2727--2746.

\bibitem{asanjarani2021survey}
\newblock{A~Asanjarani, Y~Nazarathy, and P~Taylor (2021).}
\newblock{A survey of parameter and state estimation in queues.}
\newblock{\em Queueing Systems}, 97:39--80.

\bibitem{B1961}
\newblock{P.~Billingsley (1961).}
\newblock{Statistical methods in Markov chains.}
\newblock{\em The Annals of Mathematical Statistics,} 32(10):12--40.

\bibitem{BMR2023}
\newblock{S.A.~Bodas,  M.~Mandjes and L.~Ravner (2023).}
\newblock{Statistical inference for a service system with non-stationary arrivals and unobserved balking.}
\newblock{\em arXiv preprint} arXiv:2311.16884.

\bibitem{bogachev_measure_2007}
\newblock{V.I.~Bogachev, (2007).}
\newblock{\em Measure Theory}. Springer Science \&
Business Media.

\bibitem{botta_approximation_1986}
\newblock{R.~Botta and C.~Harris (1986).}
\newblock{Approximation with generalized hyperexponential distributions: weak convergence results.}
\newblock{\em Queueing Systems}, 1:169--190.

\bibitem{CLH2023}
\newblock{X.~Chen,  Y.~Liu and G.~Hong (2023).}
\newblock{An online learning approach to dynamic pricing and capacity sizing in service systems.}
\newblock{\em  Operations Research}.

\bibitem{durrett_probability_2019}
\newblock{R.~Durret, (2019).}
\newblock{\em Probability: Theory and Examples}.  Cambridge University Press.

\bibitem{book_F1996}
{\sc T. Ferguson} (1996).
\newblock{\em A Course in Large Sample Theory}.
\newblock Chapman \& Hall, Routledge.

\bibitem{book_H2016}
\newblock{R.~Hassin (2016).}
\newblock{\em Rational Queueing}, CRC Press.

\bibitem{book_HH2003}
\newblock{R.~Hassin, M.~Haviv(2003).}
\newblock{\em To Queue or Not to Queue: Equilibrium Behavior in Queueing Systems}, Springer.

\bibitem{book_H2013}
\newblock{M.~Haviv (2013).}
\newblock{\em Queues - A Course in Queueing Theory, volume 191 of International Series in Operations Research \& Management Science.} Springer.

\bibitem{inoue2023estimating}
\newblock{Y~Inoue, L~Ravner, and M~Mandjes (2023).}
\newblock{Estimating customer impatience in a service system with unobserved  balking.}
\newblock{\em Stochastic Systems}, 13(2):181--210.

\bibitem{karlin_classification_1957}
\newblock{S.~Karlin and J.~McGregor (1957).}
\newblock{The Classification of Birth and Death Processes.}.
\newblock{\em Transactions of the American Mathematical
Society}, 86 (2): 366--400.

\bibitem{Larsen1998}
\newblock{C.~Larsen (1998). }
\newblock{Investigating sensitivity and the impact of information on pricing decisions in an M/M/$1/\infty$ queueing model.}
{\em International Journal of Production Economics, }, 56,:365--377.

\bibitem{Naor1969}
\newblock{P. Naor (1969). }
\newblock{The regulation of queue size by levying tolls.}
{\em Econometrica}, 37:15--24.

\bibitem{RS2023}
\newblock{L.~Ravner and R.I.~Snitkovsky (2023).}
\newblock{Stochastic approximation of symmetric nash equilibria in queueing games.}
\newblock{\em Operations Research}.

\bibitem{ravner_estimating_2023}
\newblock{L.~Ravner and J.~Wang (2023).}
\newblock{Estimating customer delay and tardiness sensitivity from periodic queue length observations.}
\newblock{\em Queueing Systems}, 103(3):241--274.

\bibitem{book_vdV1998}
\newblock{A. van~der Vaart (1998).}
\newblock{\em Asymptotic Statistics}, volume~3.
\newblock Cambridge University Press.

\bibitem{vdV2010}
\newblock{A. van~der Vaart (2010).}
\newblock{Time series}.
\newblock{\em VU Amsterdam, lecture notes}.

\bibitem{W2012}
\newblock{W.~Whitt (2012).}
\newblock{Fitting birth-and-death queueing models to data.}
\newblock{\em  Statistics \& Probability Letters}, 82(5):998--1004.

\end{thebibliography}
\end{document}